\documentclass[a4paper,10pt]{amsart}
\usepackage{amsfonts, amsmath, amsthm, amssymb, bbm}

\usepackage[T2A]{fontenc}
\usepackage[cp1251]{inputenc}
\usepackage[russian, english]{babel}
\usepackage{xcolor,graphics}

\newtheorem{theorem}{Theorem}[section]

\newtheorem{lemma}[theorem]{Lemma}
\newtheorem{corollary}[theorem]{Corollary}

\theoremstyle{definition}
\newtheorem{definition}[theorem]{Definition}

\theoremstyle{remark}
\newtheorem{remark}[theorem]{Remark}

\theoremstyle{remark}
\newtheorem*{claim*}{Claim}

\newcommand{\E}{\mathbb{E}}
\renewcommand{\P}{\mathbb{P}}
\newcommand{\alg}{\overline{\mathbb{Q}}}
\newcommand{\R}{\mathbb{R}}
\newcommand{\Z}{\mathbb{Z}}
\renewcommand{\C}{\mathbb{C}}
\renewcommand{\Re}{\operatorname{Re}}
\renewcommand{\Im}{\operatorname{Im}}
\newcommand{\bx}{\mathbf{x}}
\newcommand{\bw}{\mathbf{w}}
\newcommand{\bt}{\mathbf{t}}
\newcommand{\bs}{\mathbf{s}}

\newcommand{\by}{\mathbf{y}}
\newcommand{\bz}{\mathbf{z}}
\newcommand{\bu}{\mathbf{u}}

\newcommand{\be}{\boldsymbol{\eta}}
\newcommand{\dd}{{\rm d}}
\newcommand{\rv}{{\rm v}}
\newcommand{\ind}{\mathbbm{1}}

\newcommand{\h}{Q} % upper bound for height
\newcommand{\dist}{\mathop{\mathrm{dist}}\nolimits}
\newcommand{\Vol}{\mathop{\mathrm{Vol}}\nolimits}
\newcommand{\conj}[1]{\overline{#1}}

\newcommand{\eqdistr}{\stackrel{d}{=}}

\newcommand{\Lip}{\operatorname{Lip}}

\newcommand{\tc}{}
\newcommand{\tb}{}
\newcommand{\tv}{}

\title[Joint distribution of conjugate algebraic numbers]{Joint distribution of conjugate algebraic numbers: a random polynomial approach}

\keywords{Сonjugate algebraic numbers, correlations between algebraic numbers, distribution of algebraic numbers, integral polynomials, random polynomials, mixed correlation functions, real zeros, complex zeros, Coarea formula, Bombieri norm, $l_p$-balls, weighted $l_p$-norm}
\subjclass[2010]{Primary, 11N45; secondary, 11C08, 60G55, 30C15, 26C10}
\thanks{The work  was done with the financial support of the Bielefeld University (Germany) in terms of project SFB 1283. The work of the third author is supported  by the Program of the Presidium of the Russian
	Academy of Sciences ``Latest methods in mathematical modeling in the study of nonlinear dynamical systems'' 	(targeted subsidy 08-04).}

\author[F.~G\"otze]{Friedrich G\"otze}
 \address{Friedrich G\"otze, Faculty of Mathematics,
 Bielefeld University,
 P. O. Box 10 01 31,
 33501 Bielefeld, Germany}
 \email{goetze@math.uni-bielefeld.de}

\author[D.~Koleda]{Denis Koleda}
 \address{Denis Koleda, Institute of Mathematics,
 National Academy of Sciences of Belarus,
 220072 Minsk, Belarus}
 \email{koledad@rambler.ru}

\author[D.~Zaporozhets]{Dmitry Zaporozhets}
 \address{Dmitry Zaporozhets\\
 St.\ Petersburg Department of
 Steklov Institute of Mathematics,
 Fontanka~27,
  191011 St.\ Petersburg,
 Russia}
 \email{zap1979@gmail.com}

\begin{document}

\begin{abstract}

We count the algebraic numbers of fixed degree by their $\mathbf{w}$-weighted $l_p$-norm which generalizes the na{\"\i}ve height, the length, the Euclidean  and the Bombieri norms. For non-negative integers $k,l$ such that $k+2l\leq n$ and a Borel subset $B\subset \mathbb{R}\times\mathbb{C}_+^l$ denote by $\Phi_{p,\mathbf{w},k,l}(Q,B)$  the number of ordered $(k+l)$-tuples in $B$ of conjugate algebraic numbers of
degree~$n$ and $\mathbf{w}$-weighted $l_p$-norm at most $ Q$. We show that
$$
\lim_{ Q\to\infty}\frac{\Phi_{p,\mathbf{w},k,l}( Q,B)}{ Q^{n+1}}=\frac{{\mathrm{Vol}}_{n+1}(\mathbb{B}_{p,\mathbf{w}}^{n+1})}{2\zeta(n+1)}\int_B \rho_{p,\mathbf{w},k,l}(\mathbf{x},\mathbf{z}){\rm d}\mathbf{x}{\rm d}\mathbf{z},
$$
where ${\mathrm{Vol}}_{n+1}(\mathbb{B}_{p,\mathbf{w}}^{n+1})$ is the volume of the unit $\mathbf{w}$-weighted $l_p$-ball  and  $\rho_{p,\mathbf{w},k,l}$ \tc{will} denote the correlation function of $k$ real and $l$ complex zeros of the random polynomial $\sum_{j=1}^n \frac{\eta_j}{w_j} z^j$, where $\eta_j $  are i.i.d.  random variables  with density $c_p e^{-|t|^p}$ for $0<p<\infty$ and with constant density on $[-1,1]$ for $p=\infty$. If the boundary of $B$ is of Lipschitz type, we also estimate the rate of convergence.
We give an explicit formula for $\rho_{p,\mathbf{w},k,l}$\tc{,} which in the case $k+2l=n$ has a very simple form. To this end, we obtain a general formula for the correlations between real and complex zeros of a random polynomial with arbitrary independent absolutely continuous  coefficients.

% simplifies to
%$$
%\rho_{p,\bw,n-2l,l}=\frac{2}{(n+1)\Vol_{n+1}(\mathbb{B}_{p,\bw}^{n+1})}\,\frac{\sqrt{|\mathrm{D}[q]|}\phantom{1^n}}{(l_{p,\bw}[q])^{n+1}},
%$$
%where $q$ is the monic polynomial whose zeros are the arguments of the correlation function $\rho_{p,\bw,n-2l,l}$
%and $\mathrm{D}[q]$ denotes its discriminant.

%As an application, we show that with respect to the Bombieri 2-norm, i.e. for $p=2$ and $\bw=\mathbf{b}:=\big(\binom{n}{j}^{-1/2}\big)_{j=0}^n$, the asymptotic density of real algebraic numbers coincides with the normalized Cauchy
%density up to a constant factor depending on $n$ only:
%$$
%\rho_{2,\mathbf{b},1,0}(x)= c_n\cdot\frac{1}{\pi(1+x^2)}.
%$$
\end{abstract}

\maketitle

\section{Introduction}
What is the distribution of algebraic numbers of a given degree $n$? To answer this question, we first need to define a suitable notion for it. For discrete sets like e.g. infinite subsets of
integers
one is lead to consider asymptotic relative densities of such sets in intervals of integers $[1,N]$  in the large $N$ limit.
Clearly, this direct approach does not work for algebraic numbers (real or complex), since any domain contains an infinite number of them (even for a fixed degree $n$).
A possible solution is \tc{the following classical ordering} of  algebraic numbers by the concept of height.
Let $\alg$ denote the field of (all) algebraic numbers over $\mathbb{Q}$. A function $h:\alg\to\R_+$ is called a \emph{height function} if for any $n\in\Z_+$ and $\h>0$ there are only finitely many algebraic numbers
$\alpha$ of degree~$n$ such that $h(\alpha)\leq \h$. Note that usually it is required (and we will always assume this) that $h(\alpha')=h(\alpha)$ for all conjugates of~$\alpha$.

Having defined $h$, we are interested in the asymptotic number of $\alpha\in\alg$ of degree~$n$ lying in a given subset $B$ of $\R$ or $\C$ such that $h(\alpha)\leq\h$ as $\h\to\infty$. More generally, for $k=1,\dots,n$,
one would like to determine  the asymptotic behaviour of the number of $k$-tuples $(\alpha_1,\dots,\alpha_k)\in B_1\times\dots\times B_k$ of \emph{conjugate} algebraic numbers of degree $n$ and height at most $\h$ as
$\h\to\infty$.

In this paper, we consider \tc{the} so-called weighted $l_p$-heights ($p\in(0,\infty]$).

We would like to emphasize that throughout this paper the degree $n$ of polynomials, algebraic numbers, etc. is {\it fixed}.

Section \ref{sc-bas-def} contains some basic notation. In Section \ref{sc-alg-numb} we describe the problem of counting vectors with algebraic coordinates and formulate the main number-theoretical results of the paper.
Section \ref{sc-rnd-poly} gives a necessary account of related topics in random polynomials. There we formulate the principal result relating  distributions of zeros of random polynomials and algebraic numbers. In
Section \ref{sc-corr-func} we state several explicit formulae for a function which plays the role of the joint distribution functions for conjugate algebraic numbers. Sections \ref{sc-proof-alg-n}--\ref{sc-proof-add-ident} contain the proofs of our statements.

\section{Basic definitions}\label{sc-bas-def}

Given a polynomial $q(z):=a_0+a_1z+\dots+a_nz^n$ and a vector of positive weights $\bw=(w_0, w_1,\dots,w_n)$ define the $\bw$-weighted $l_{p,\bw}$-norm of $q$ as
$$
l_{p,\bw}[q]:=
\left\{
  \begin{array}{ll}
    \left(\sum_{i=0}^{n}|w_i a_i|^p\right)^{1/p}, & 0<p<\infty; \\
    \max_{0\le i\le n} w_i |a_i|, & p=\infty.
  \end{array}
\right.
$$
Although it is not a real norm for $0<p<1$, all our results remain true in this case as well.

In the non-weighted case $\bw=\mathbf{1}$, this notion generalizes the na\"{\i}ve height ($p=\infty$), the length ($p=1$), and the {Euclidean} norm ($p=2$):
$$
H[q] := \max_{0\le i\le n} |a_i|,\quad L[q] := \sum_{i=0}^n |a_i|, \quad\|q\|:=\left(\sum_{i=0}^{n}a_i^2\right)^{1/2}.
$$
An important weighted example is the Bombieri $p$-norm:
$$
[q]_p:=\left(\sum_{i=0}^{n}\binom{n}{i}^{1-p}|a_i|^p\right)^{1/p}.
$$

Denote by $\mathbb{B}_{p,\bw}^{n+1}$ the $(n+1)$-dimensional unit $l_{p,\bw}$-ball:
\begin{equation}\label{eq-unit-ball}
\mathbb{B}_{p,\bw}^{n+1}:=\left\{(a_0,\dots,a_n)\in\R^{n+1}\colon \sum_{i=0}^{n}|w_i a_i|^p\leq1 \right\}.
\end{equation}
Using the well-known formula for the volume of the (non-weighted) unit $l_p$-ball $\mathbb{B}_{p}^{n+1}$ we have
\begin{equation}\label{1453}
\Vol_{n+1}(\mathbb{B}_{p,\bw}^{n+1})=\frac{\Vol_{n+1}(\mathbb{B}_{p}^{n+1})}{w_0 w_1 \dots w_n}=
\left\{
\begin{array}{ll}
\frac{2^{n+1} \Gamma\left(1+\frac1p\right)^{n+1}}{w_0 w_1 \dots w_n \Gamma\left(1+\frac{n+1}p\right)}, & p<\infty, \\
\frac{2^{n+1}}{w_0 w_1 \dots w_n }, & p=\infty.
\end{array}
\right.
.
\end{equation}

Let $\mathcal{P}_{p,\bw}(\h)$ denote the class of integral polynomials (polynomials with integer coefficients) of degree $n$ and the $l_{p,\bw}$-height at most $\h$:
$$
\mathcal{P}_{p,\bw}(\h):=\{q\in\mathbb{Z}[z]\colon\deg[q]=n, \, l_{p,\bw}[q]\leq\h\}.
$$
We say that an integral polynomial  is \emph{prime}, if it is irreducible over $\mathbb{Q}$, primitive (the greatest common divisor of its coefficients equals 1), and its leading coefficient is positive. Denote by
$\mathcal{P}_{p,\bw}^{*}(\h)$ the class of prime polynomials from $\mathcal{P}_{p,\bw}(\h)$:
$$
\mathcal{P}^*_{p,\bw}(\h):=\{q\in\mathcal{P}_{p,\bw}(\h)\colon q \text{ is prime}\}.
$$

The \emph{minimal polynomial} of an algebraic number $\alpha$ is the (unique) prime polynomial $q$ such that $q(\alpha)=0$. We put by definition
$$
l_{p,\bw}[\alpha]:=l_{p,\bw}[q].
$$
The other roots of $q$ are called \tc{the} (algebraic) \emph{conjugates} of $\alpha$.

\section{Distribution of algebraic numbers}\label{sc-alg-numb}

We would like to study the joint distribution of several conjugate algebraic numbers of degree $n$ and bounded height. What is the natural configuration space for this problem? Consider a  prime polynomial of degree $n$.
Some of its zeros are real, and the rest are symmetric with respect to the real line. Thus we may neglect the zeros lying in $\C_-$. Fix some integer numbers $k,l\geq0$ such that $0<k+2l\leq n$.

For a measurable set $B\subset \R^k\times\C_+^l$ and the height function $l_{p,\bw}$ denote by $\Phi_{p,\bw,k,l}(\h,B)$ the \tc{number of  $(k+l)$-tuples} $(\alpha_1,\dots,\alpha_k,\beta_1,\dots,\beta_l)\in B$ of \emph{distinct}
numbers such that for some $q\in\mathcal{P}_{p,\bw}^{*}(\h)$ it holds
$$
q(\alpha_1)=\dots=q(\alpha_k)=q(\beta_1)=\dots=q(\beta_l)=0.
$$
Essentially $\Phi_{p,\bw,k,l}(\h,B)$ denotes the number of ordered $(k+l)$-tuples in $B$ of conjugate algebraic numbers of degree $n$ and height $l_{p,\bw}$ at most $Q$.

We always assume that $B$ is measurable and its boundary $\partial B$ has Lebesgue measure $0$. Our aim is to show that there exists a non-trivial limit
\begin{equation}\label{1848}
\lim_{\h\to\infty}\frac{\Phi_{p,\bw,k,l}(\h,B)}{\h^{n+1}}
\end{equation}
and to find its exact value.

\begin{theorem}\label{2235}
For $p\in(0,\infty]$, a fixed positive vector $\bw=(w_0,w_1,\dots,w_n)$, and some integer numbers $k,l\geq0$ such that $0<k+2l\leq n$, there exists a function $\rho_{p,\bw,k,l}:\R^k\times\C_+^l\to\R_+$ such that for any  measurable $B\subset
\R^k\times\C_+^l$ with the boundary having Lebesgue measure $0$ we have
$$
\lim_{\h\to\infty}\frac{\Phi_{p,\bw,k,l}(\h,B)}{\h^{n+1}}=\frac{\Vol_{n+1}(\mathbb{B}_{p,\bw}^{n+1})}{2\zeta(n+1)}\int_B \rho_{p,\bw,k,l}(\bx,\bz)\dd\bx\dd\bz,
$$
where the formula for $\Vol_{n+1}(\mathbb{B}_{p,\bw}^{n+1})$ is given in~\eqref{1453}. The expression for $\rho_{p,\bw,k,l}(\bx,\bz)$ will be given in Corollary~\ref{2111}.
\end{theorem}
{Provided that $\partial B$ is smooth enough we are able to estimate the rate of convergence in~\eqref{1848}}  as well. To clarify what we mean by ``smooth enough'' we need the following definition (see~\cite[Definition 2.2]{mW12}). 
\begin{definition}\label{2104}
We say that  $S\subset\R^d$ is  of Lipschitz class $(M,L)$ ($S\in\operatorname{Lip}(M,L))$ if  there exist $M$ maps
$\phi_1, \dots, \phi_M : [0, 1]^{d-1} \to\R^d$ satisfying a Lipschitz condition
\[
|\phi_i(x) - \phi_i(y)| \le L|x - y| \text{ for } x, y \in [0, 1]^{d-1}, \quad  i = 1, \dots,M,
\]
such that $S$ is covered by the images of the maps $\phi_i$. 
\end{definition}
 \tc{ \begin{remark}\label{2006}
     It is clear that if in the Definition \tb{\ref{2104}}  we replace the domain $[0,1]^{d-1}$  by  an arbitrary parallelepiped, then $S$ is still of Lipschitz class with some different parameter $L$.
 \end{remark}
	}

This definition \tb{can be naturally extended to}  subsets of $\R^{d_1}\times\C^{d_2}$ \tb{as well by considering} them as images under the standard isometry between $\R^{d_1}\times\C^{d_2}$ and $\R^{d_1+2d_2}.$

\begin{theorem}\label{1224}
Suppose that the boundary of $B$ is of Lipschitz class $(M,L)$. Then,
$$
\left|\frac{\Phi_{p,\bw,k,l}(\h,B)}{\h^{n+1}}-\frac{\Vol_{n+1}(\mathbb{B}_{p,\bw}^{n+1})}{2\zeta(n+1)}\int_B \rho_{p,\bw,k,l}(\bx,\bz)\dd\bx\dd\bz\right|\leq \tc{ \begin{cases}
C\frac {\log\h}\h, & n=2, \ l=0, \\
C\frac 1\h, & \text{otherwise},
\end{cases}}
$$
where $C$ depends on $n, M, L, \bw$ only.

\end{theorem}

For the na\"{\i}ve height ($p=\infty$, $w_i = 1$), the problem of finding the limit in~\eqref{1848} was solved for real ($k=1,l=0$) and complex ($k=0,l=1$) algebraic numbers, see~\cite{dK14} and~\cite{GKZ15}, and also for tuples of real conjugate algebraic numbers ($k$ is arbitrary, $l=0$), see \cite{GKZ15c}.

Unfortunately, our approach cannot cover the case of the Mahler measure. The Mahler measure (in particular, in the form of the Weil height) has many applications in algebraic number theory. Counting algebraic numbers
and points with respect to the Weil height and its generalisations has been intensively studied.
See the papers \cite{MV08}, \cite{MV07}, \cite{mW16}, \cite{GG16} for results in this direction and related references.

It turns out that the function $\rho_{p,\bw,k,l}$ coincides with the correlation function of the zeros of some specific random polynomial. To formulate the result we first recall some essential notions.

\section{Zeros of random polynomials}\label{sc-rnd-poly}
 Let $\xi_0,\xi_1,\dots,\xi_{n}$ be independent \emph{real-valued} random variables with bounded probability density functions $f_0,\dots,f_n$. Consider the random polynomial defined {as}
\begin{equation}\label{1914}
G(z):=\xi_0+\xi_1z+\dots+\xi_nz^n,\quad z\in\C.
\end{equation}
With probability one, all zeros of $G$ are simple, see~\cite{eB62}. Denote by $\mu$ the empirical measure counting {the} zeros of $G$:
\[
\mu:=\sum_{z:G(z)=0}\delta_z,
\]
where $\delta_z$ is the unit point mass at $z$. The random measure $\mu$ may be regarded as a  random point process on $\mathbb{C}$. {A} natural way of describing the distribution of a point process is via its \emph{correlation functions}.
Since the coefficients of $G$ are real, its  zeros are symmetric with respect to the real line, and some of them possibly are real. Therefore, the natural configuration space for the point process $\mu$
must be a ``separated'' union $\C_+\cup\R$ with topology induced by the union of topologies in $\C_+$ and $\R$. Instead of considering the correlation functions of the process on $\C_+\cup\R$, {an equivalent way is to
investigate the} \emph{mixed $(k,l)$-correlation functions} (see~\cite{TV15}).
We call functions   $\rho_{k,l}\,:\,\R^k\times\mathbb{C}_+^l \to\R_+$, where $0<k+2l\le n$, {mixed $(k,l)$-correlation functions of the zeros of $G$}, if for any family of mutually disjoint Borel subsets
$B_1,\dots,B_k\subset\R$ and $B_{k+1},\dots,B_{k+l}\subset\C_+$,
\begin{equation}\label{eq-rhokl}
\E\,\left[\prod_{i=1}^{k+l}\mu(B_i)\right]=\int_{B_1}\dots\int_{B_{k+l}}\,\rho_{k,l}(\bx,\bz)\,\dd\bx\,\dd\bz,
\end{equation}
{where $\E$ denotes the expectation with respect to the joint distribution of $\xi_0,\xi_1,\dots,\xi_{n}$.}
Here and subsequently, we write
$$
\bx:=(x_1,\dots,x_k)\in\R^k,\quad \bz:=(z_1,\dots,z_l)\in\C_+^l.
$$

The most intensively studied class of random polynomials are \emph{Kac polynomials}, when $\xi_i$'s are i.i.d.  Sometimes the i.i.d. coefficients are considered with some non-random weights $c_i$'s. The common
examples are \emph{flat or Weil polynomials} ({with} $c_i=\sqrt{1/i!}$) and \emph{elliptic polynomials} ({with} $c_i=\sqrt{\binom{n}{i}}$).

The $(1,0)$-correlation function $\rho_{1,0}$ is called { a density of real zeros.  Integrated over $\R$} it gives the average number of real zeros of $G$. The asymptotic properties of this quantity as $n\to\infty$ have
been intensively  studied for many years,  mostly for Kac polynomials; see the historical background  in~\cite{BS86} and the survey of the most recent results in~\cite{TV15}. We just mention some {contributions here
like}:~\cite{mK43}, \cite{EO56}, \cite{dS69}, \cite{IM71-1}, \cite{NZ09}, \cite{LPX15}, \cite{kS16b}, \cite{tB16}, \cite{FK18}.

{Similarly, $\rho_{0,1}$ is called a density of complex zeros being an expectation of the empirical measure $\mu$ counting non-real zeros.} {Its limit behaviour as $n\to\infty$ is of a great interest as well},
see~\cite{SS62}, \cite{IZ97}, \cite{IZ13}, \cite{KZ14b}, \cite{iP16}, \cite{PR16}, \cite{tB17},   and the references given there.

{There are comparatively few papers on higher-order correlation functions of zeros. Well-known results are due to Bleher and Di~\cite{BD97}, \cite{BD04} who studied the correlations between real zeros for elliptic and
Kac polynomials, and to Tao and Vu~\cite{TV15} who proved  asymptotic universality for the mixed correlation functions for elliptic, Weil, and Kac polynomials under some moment conditions on $\xi_i$.}

Our next result connects the limit density of tuples of conjugate algebraic numbers with the correlation function of zeros of the following random polynomial.

\begin{theorem}\label{1558}
Let $p\in(0,\infty]$ and let $\eta_0,\eta_1,\dots,\eta_n$ be i.i.d. real random variables {with a probability density function} given by
\begin{equation}\label{1422}
f(t):=
\left\{
  \begin{array}{ll}
    \frac{1}{2\Gamma\left(1+\frac1p\right)}e^{-|t|^p}, & p<\infty, \\
    \frac{1}{2}\ind_{[-1,1]}(t), & p=\infty.
  \end{array}
\right.
\end{equation}
Consider the random polynomial defined as
\begin{equation}\label{1956}
G_{p,\bw}(z):=\sum_{i=0}^{n}\frac{\eta_i}{w_i} z^i.
\end{equation}
Then, the mixed $(k,l)$-correlation function of zeros of $G_{p,\bw}$ coincides with the function $\rho_{p,\bw,k,l}$ from Theorem~\ref{2235}.
\end{theorem}
The exact formula for $\rho_{p,\bw,k,l}$  is given in Section~\ref{sc-corr-func}. {Let us now} formulate one important special case.

\subsection{Bombieri 2-Norm}
The next theorem shows that the way of counting algebraic numbers with respect to the Bombieri 2-norm is in some sense the most natural.  It has been  shown in~\cite{EK95} that in the case of Bombieri 2-norm the zeros of
the corresponding random polynomial (sometimes called the \emph{elliptic} random polynomial) have a very simple density.
\begin{theorem}[Edelman--Kostlan \cite{EK95}]\label{1559}
Assume that
$
\bw=\left(\binom{n}{i}^{-1/2}\right)_{i=0}^n.
$
Then
$$
\rho_{2,\bw,1,0}(x)=\frac{\sqrt{n}}{\pi(1+x^2)}.
$$
\end{theorem}
Thus for any degree $n$ the asymptotic density of algebraic numbers counted with respect to Bombieri 2-norm coincides with the normalized Cauchy density. In particular,  Theorem~\ref{2235} implies the following.
\begin{corollary}
We have
$$
\lim_{\h\to\infty}\frac{\Phi_{2,\bw,1,0}(\h,B)}{\h^{n+1}}=\frac{\sqrt{n}\Vol_{n+1}(\mathbb{B}_{2,\bw}^{n+1})}{2 \pi\, \zeta(n+1)} \int_B \frac{\dd x}{1+x^2}.
$$
\end{corollary}
The volume $\Vol_{n+1}(\mathbb{B}_{2,\bw}^{n+1})$ can be calculated as
\[
\Vol_{n+1}(\mathbb{B}_{2,\bw}^{n+1}) = \frac{\pi^{\frac{n+1}{2}}}{\Gamma\!\left(1+\frac{n+1}{2}\right)} \sqrt{\prod_{i=0}^n \binom{n}{i}} =
\frac{(\pi\, n!)^{\frac{n+1}{2}}}{\Gamma\!\left(1+\frac{n+1}{2}\right) \prod_{i=0}^n i!}.
\]

\section{General Formula for $\rho_{p,\bw,k,l}$}\label{sc-corr-func}

Recall that $f_0,\dots,f_n$ {denote} the probability density functions of the coefficients $\xi_0,\dots,\xi_n$ of $G$ and $\rho_{k,l}$ denotes the mixed correlation function of its zeros; see~\eqref{1914} and~\eqref{eq-rhokl}.  For $m=1,\dots,n,$ consider a function $\rho_m:\C^m\to\R$ defined {as}
\begin{multline}\label{eq-rho-b}
\rho_m(z_1,\dots,z_m):= \prod_{1\le i < j \le m} |z_i - z_j|\\
\times\int\limits_{\mathbb{R}^{n-m+1}} \prod_{i=0}^{n}f_i\left(\sum_{j=0}^{n-m}(-1)^{m-i+j}\sigma_{m-i+j}(z_1,\dots,z_m)t_j\right) \prod_{i=1}^m \Bigg|\sum_{j=0}^{n-m} t_j z_i^j\Bigg|\dd t_0\dots \dd t_{n-m},
\end{multline}
where we used the following notation  for the elementary symmetric polynomials:
\begin{equation}\label{1738}
\sigma_i(z_1,\dots,z_m) :=
\left\{
  \begin{array}{ll}
    1, & \hbox{if}\quad i=0,\\
    \sum_{1\le j_1 < \dots < j_i\le m} z_{j_1} z_{j_2} \dots z_{j_i}, & \hbox{if}\quad 1\leq i\le m, \\
    0, & \hbox{otherwise.}
  \end{array}
\right.
\end{equation}
It is tacitly assumed that the arguments of $f_i$'s are well-defined (i.e., real):
we shall restrict ourselves \tc{by considering only those points} $(z_1,\dots,z_m)$  such that all symmetric functions of them are real.

Introduce as well the absolute value of the  Vandermonde determinant:
\begin{equation}\label{1330}
  \rv_m(z_1,\dots,z_m):=\prod_{1\le i < j \le m} |z_i - z_j|.
\end{equation}

It was proved in~\cite{GKZ15b} that the correlation functions of real zeros of $G$ are given by~\eqref{eq-rho-b}:
$$
\rho_{k,0}(\bx)=\rho_{k}(\bx)
$$
for all $\bx\in\R^k$. The following theorem {generalises} this relation to all mixed $(k,l)$-correlation functions.

\begin{theorem}\label{thm-corr-rel}
For all $(\bx,\bz)\in\R^k\times\C_+^l$,
\begin{equation}\label{eq-corr-rel}
\rho_{k,l}(\bx,\bz) = 2^l \rho_{k+2l}(\bx,\bz,\bar\bz),
\end{equation}
where $\rho_{k+2l}$ is defined in~\eqref{eq-rho-b}. Specifically,
\begin{multline}\label{2008}
  \rho_{k,l}(\bx,\bz) =2^l\rv_{k+2l}(\bx,\bz,\bar\bz)\int\limits_{\mathbb{R}^{n-k-2l+1}} \prod_{i=0}^{n}f_i\left(\sum_{j=0}^{n-k-2l}(-1)^{k-i+j}t_j\sigma_{k+2l-i+j}(\bx,\bz,\bar\bz)\right)  \\
    \times \prod_{i=1}^k \Bigg|\sum_{j=0}^{n-k-2l} t_j x_i^j\Bigg|\cdot \prod_{i=1}^l \Bigg|\sum_{j=0}^{n-k-2l} t_j z_i^j\Bigg|^2\dd t_0\dots \dd t_{n-k-2l}.
\end{multline}
\end{theorem}
The proof of Theorem~\ref{thm-corr-rel} is  given in Section~\ref{1245}.

%It is interesting to
{Note that  the correlations between real zeros and the correlations between complex zeros are essentially given by the {\it same} function $\rho_m$. In particular, $\rho_2$ provides a}  formula for the density of
complex zeros as well as for the two-point correlation function of real zeros:
$$
\rho_{0,1}(z)=2\rho_2(z,\bar z),\quad \rho_{2,0}(x,y)=\rho_2(x,y).
$$

Let us give some examples. Taking $k=1,l=0$ in~\eqref{2008} yields {a} formula for the density of real zeros.
\begin{corollary}
We have
\begin{align*}
\rho_{1,0}(x) =\int\limits_{\mathbb{R}^{n}} \prod_{i=0}^{n}f_i\left(t_{i-1}-xt_i\right)\Bigg|\sum_{j=0}^{n-1} t_j x^j\Bigg|\dd t_0\dots \dd t_{n-1},
\end{align*}
where we set $t_{-1}=t_n:=0$.
\end{corollary}
This formula (with different notations) was obtained in~\cite{dZ05}.

Taking $k=0,l=1$ yields {a} formula for the density of complex zeros.
\begin{corollary}
We have
\begin{align*}
\rho_{0,1}(z)=4|\Im z|\int\limits_{\mathbb{R}^{n-1}} \prod_{i=0}^{n}f_i(t_{i-2}-2t_{i-1}\Re z+t_i|z|^2)
\Bigg|\sum_{j=0}^{n-2} t_j z^j\Bigg|^2\dd t_0\dots \dd t_{n-2},
\end{align*}
where we set $t_{-2}=t_{-1}=t_{n-1}=t_n:=0$.
\end{corollary}
This formula (with different notations) was obtained in~\cite{dZ04}.

Taking $k=n-2l$ we obtain the (non-normalized) joint density of all zeros given that $G$ has exactly $n-2l$ real zeros.
\begin{corollary}
	We have
\begin{align}\label{1638}
\rho_{n-2l,l}(\bx,\bz)=2^l\rv_{n}(\bx,\bz,\bar\bz)\int\limits_{\R }|t|^n \prod_{i=0}^{n}f_i\left((-1)^{n-i}t\sigma_{n-i}(\bx,\bz,\bar\bz)\right)\dd t.
\end{align}
\end{corollary}
Recall that $\mu$ denotes the empirical measure counting the zeros of $G$. It is easy to derive {from~\eqref{eq-rhokl}} that
\begin{align*}
\E\left[\frac{\mu(\R)!}{(\mu(\R)-n+2l)!}\frac{\mu(\C_+)!}{(\mu(\C_+)-l)!}\right]=\int_{\R^{n-2l}}\int_{\C^{l}_+}\rho_{n-2l,l}(\bx,\bz)\dd\bx\dd\bz,
\end{align*}
where we used the convention {$0!:=1$ and $q!:=\infty$ for any integer $q < 0$}.
Since with probability one $\mu(\R)+2\mu(\C_+)=n$, the random variable under expectation is non-zero if and only if $\mu(\R)=n-2l$. Thus we obtain the probability that $G$  has exactly $n-2l$ real zeros.
\begin{corollary}
	We have
\begin{multline*}
\P[\mu(\R)=n-2l]=\frac{2^l}{l!(n-2l)!}\int_{\R^{n-2l}}\int_{\C^{l}_+}\rv_{n}(\bx,\bz,\bar\bz)\\\times\int\limits_{\R }|t|^n \prod_{i=0}^{n}f_i\left((-1)^{n-i}t\sigma_{n-i}(\bx,\bz,\bar\bz)\right)\dd t\dd\bx\dd\bz.
\end{multline*}
\end{corollary}
This formula has been obtained earlier in~\cite{dZ04}.

From Theorems \ref{1558} and \ref{thm-corr-rel}, the following statement {immediately follows}.
\begin{corollary}\label{2111}
\tc{\tb{An} explicit representation for the integrand $\rho_{p,\bw,k,l}(\bx,\bz)$ of Theorem~\ref{2235} can be obtained from~\eqref{2008} } \tb{by choosing the densities of coefficients as}
\[
f_i(t) =
\left\{
  \begin{array}{ll}
    \frac{w_i}{2\Gamma\left(1+\frac1p\right)}e^{-|w_i t|^p}, & p<\infty, \\
    \frac{w_i}{2}\ind_{[-1,1]}(w_i t), & p=\infty.
  \end{array}
\right.
\]
\end{corollary}

In the case $k=n-2l$ the formula for $\rho_{p,\bw,k,l}(\bx,\bz)$ can be considerably simplified \tb{as shown in the} next theorem. Before formulating it we introduce the following notion: for an arbitrary monic polynomial $q(z):=(z-z_1)\dots(z-z_n)$ define its discriminant as
$$
\mathrm{D}[q]=\mathrm{D}[z_1,\dots,z_n]:=\prod_{i<j}(z_i-z_j)^2.
$$
\begin{theorem}\label{1600}
We have
$$
\rho_{p,\bw,n-2l,l}=\frac{2^{l+1}}{(n+1)\Vol_{n+1} (\mathbb{B}_{p,\bw}^{n+1})}\, \frac{\sqrt{|\mathrm{D}[q]|}\phantom{^{n+1}}}{(l_{p,\bw}[q])^{n+1}},
$$
where $q$ is the monic polynomial whose zeros are the arguments of $\rho_{p,\bw,n-2l,l}$:
$$
q(z):=(z-x_1)\dots(z-x_k)(z-z_1)(z-\overline{z_1})\dots(z-z_l)(z-\overline{z_l}).
$$
\end{theorem}
The proof will be presented in Section~\ref{sc-proof-add-ident}.

\section{Proof of Theorems \ref{2235}, \ref{1224} and \ref{1558}}\label{sc-proof-alg-n}

\subsection{Methods of Proof: Counting Integer Points}

We reduce counting algebraic numbers to counting corresponding minimal polynomials represented by the vectors of their coefficients. So we need to formulate some statements about counting integer points in multidimensional regions.

Estimating the number of integer points in a region  by its measure  is a well-known idea. The most <<ancient>> publication (that relates  integer points counting to the volume of a region), which the authors are aware
of, is a result by Lip\-schitz~\cite{rL65}. See as well the  classical monograph by Bachmann \cite[pp. 436--444]{pB94} (in particular, formulas~(83a) and~(83b) on pages 441--442).
There are a number of papers generalizing such estimates to arbitrary lattices, see e.g. \cite{MV07} and \cite{BW14}.

For a Borel set $A\subset\R^d$ denote by $\lambda(A)$ the number of points in $A$ with integer coordinates, and by $\lambda^*(A)$ the number of points in $A$ with coprime integer coordinates:
\[
\lambda(A) = \#\left(A\cap \mathbb{Z}^d\right),
\]
\[
\lambda^*(A) = \#\left\{(x_1,\dots,x_d)\in A\cap \mathbb{Z}^d : \gcd(x_1,\dots,x_d)=1\right\}.
\]
For a real number $r$ and a set $S\subset \mathbb{R}^d$ let
\[
rS = \{r\bx : \bx\in S \}.
\]

\begin{lemma}\label{lm-bas-count}
Let $d\ge 2$ be an integer. Let $A\subset\mathbb{R}^d$ be a fixed bounded region. If the boundary $\partial A$ of $A$ has Lebesgue measure 0, then
\[
\lim_{Q\to\infty} \frac{\lambda(QA)}{Q^d} = \Vol_d(A).
\]
\end{lemma}
Note that the requirement of $A$ to be just Lebesgue measurable doesn't suffice, and the boundary $\partial A$ must be of Lebesgue measure $0$ to ensure the existence of the limit. For example, if one takes $A$ to be the set of points in
$[0,1]^d$ with rational coordinates, then for any positive integer $Q$ the set $QA$ contains $\approx Q^d$ integer points, but $A$ has  Lebesgue measure 0. Notice that in this case $\partial A=[0,1]^d$.
\begin{proof}
The lemma can be easily proved if one considers coverings of $A$ by $d$-dimensional cubes with edge $Q^{-1}$. See, for example, \cite[Chapter VI \S 2]{sL70}. Note that none of the conditions of the lemma  can be
omitted.
\end{proof}
Lemma \ref{lm-bas-count}  provides no estimates for the rate of convergence. Additionally, in this lemma the region $A$ is kept  fixed and therefore cannot depend on~$Q$.
To avoid all these restrictions
one needs to restrict oneself to a suitable class of regions. See \cite{hD51}, \cite{sL70}, \cite{pS95}.
One way is to employ the class provided in Daveport's paper \cite{hD51}. However, for our goals it is more natural to follow the approach used in~\cite{mW10}. To this end, we will need a notion of the Lipschitz class which was introduced in Section~\ref{sc-alg-numb}.

\begin{lemma}\label{lm-prim-pnt}
	Consider a bounded region $A\subset\mathbb{R}^d$, $d\geq2$, such that the boundary $\partial A$ of $A$ is in $\Lip(M,L)$. 
	Then $A$ is Lebesgue measurable and 
	\begin{equation}\label{1317}
	\left|\frac{\lambda(QA)}{\h^d}-\Vol_d(A)\right|\le \frac{C}{\h},
	\end{equation}
	where $C$ depends on $d,L,M$ only.
\end{lemma}
\begin{proof}
The proof follows directly from~\cite[Theorem~5.4]{mW10}.
\end{proof}

Now we need to adapt the latter two lemmas to the integer points with coprime coordinates.

\begin{lemma}\label{lm-prim-count}
Under the assumptions of Lemma~\ref{lm-bas-count} we have
\begin{equation}\label{2155}
\lim_{Q\to\infty} \frac{\lambda^*(QA)}{Q^d} = \frac{\Vol_d(A)}{\zeta(d)},
\end{equation}
and under the assumptions of Lemma~\ref{lm-prim-pnt} we have
\begin{equation}\label{1318}
\left|\frac{\lambda^*(QA)}{\h^d}-\frac{\Vol_d(A)}{\zeta(d)}\right|\le C\frac{\log^{\chi_{d,0}}\h}{\h},
\end{equation}
where  $C$ depends on $d,L,M$ only and
\begin{equation}\label{2151}
\chi_{n,l}:=\left\{
\begin{array}{ll}
1, & n=2, \ l=0; \\
0, & \text{otherwise}.
\end{array}
\right.
\end{equation}
\end{lemma}
\begin{proof}
Using the inclusion-exclusion principle (Moebius inversion) one can easily show that
\[
\lambda^*(QA) = \sum_{r=1}^{[QN]+1} \mu(r)\, \lambda\!\left(\frac{Q}{r} A\right),
\]
where $N$ is a positive number such that $A\subseteq[-N,N]^d$, and $\mu(\cdot)$ is the Moebius function.

Let
\[
\theta_A(Q) := \frac{\lambda(QA)}{Q^d} - \Vol_d(A).
\]
Then we have
\[
\lambda^*(QA) = Q^d \left(\sum_{r=1}^{[QN]+1} \frac{\mu(r)}{r^d} \Vol_d(A) + \sum_{r=1}^{[QN]+1} \frac{\mu(r)}{r^d}\, \theta_A\!\left(\frac{Q}{r} \right)\right).
\]
From the well-known equality $\zeta(d)^{-1}=\sum_{r=1}^{\infty}\mu(r)r^{-d}$, the first sum can be estimated as
\begin{equation}\label{1344}
\left|\sum_{r=1}^{[QN]+1} \frac{\mu(r)}{r^d} - \frac{1}{\zeta(d)}\right| = \left|\sum_{r=[QN]+2}^{\infty} \frac{\mu(r)}{r^d}\right| \le \int_{QN}^\infty \frac{dx}{x^d} = \frac{1}{(d-1)(QN)^{d-1}}.
\end{equation}
From Lemma \ref{lm-bas-count} we infer that
for any $\epsilon>0$ there exists $Q_0 = Q_0(A,\epsilon)$ such that $|\theta_A(Q)|\le \epsilon$ for all $Q\ge Q_0$,
and there exists finite $\Theta_A=\max_{Q\ge 0} |\theta_A(Q)|$.
Hence
\[
\left|\sum_{r=1}^{[QN]+1} \frac{\mu(r)}{r^d}\, \theta_A\!\left(\frac{Q}{r} \right)\right| \le \epsilon \sum_{1\le r\le \frac{Q}{Q_0}} \frac{1}{r^d} +
\Theta_A\sum_{r>\frac{Q}{Q_0}} \frac{1}{r^d} \le \zeta(d)\,\epsilon + \frac{\Theta_A}{(d-1)} \left[\frac{Q}{Q_0(\epsilon)}\right]^{1-d}.
\]
Therefore, for any $\epsilon>0$
\[
\lim_{Q\to\infty} \left|\sum_{r=1}^{[QN]+1} \frac{\mu(r)}{r^d}\, \theta_A\!\left(\frac{Q}{r} \right)\right| \le \zeta(d)\,\epsilon,
\]
which proves~\eqref{2155}.

{To prove~\eqref{1318}, note that the bound \eqref{1317} in Lemma~\ref{lm-prim-pnt} can be written as
\[
|\theta_A(Q)|\le \frac{C}{Q},
\]
where the constant $C$ is the same as in \eqref{1317}.
Hence,
\[
\left|\sum_{r=1}^{[QN]+1} \frac{\mu(r)}{r^d}\, \theta_A\!\left(\frac{Q}{r} \right)\right|
\le \frac{C}{Q} \sum_{r=1}^{[QN]+1} \frac{1}{r^{d-1}}
\le \begin{cases}
C Q^{-1} \zeta(d-1), & d\ge 3,\\
C Q^{-1} \left(\log([QN]+1)+1\right), & d=2.
\end{cases}
\]
Now joining the latter estimate with the bound \eqref{1344}
and suitably choosing a new constant $C$, we obtain \eqref{1318}.}
\end{proof}

\subsection{Proof of Theorems~\ref{2235}, \ref{1224} and \ref{1558}}

Let us describe how to calculate $\Phi_{p,\bw,k,l}(\h,B)$.

Given a function $g:\C\to\R$ and a Borel subset $B\subset\R^k\times \C^l$ denote by $\mu_{g,k,l}(B)$ the number of ordered \tc{$(k+l)$}-tuples $(x_1,\dots,x_k,z_1,\dots,z_l)\in B$ of \emph{distinct}  numbers such that
$$
g(x_1)=\dots=g(z_l)=0.
$$
For any algebraic number its minimal polynomial is prime, and any prime polynomial is a minimal polynomial for some algebraic number. Therefore we have
\begin{equation}\label{1838}
\Phi_{p,\bw,k,l}(\h,B)= \sum_{q\in\mathcal{P}_{p,\bw}^*(\h)}\mu_{q,k,l}(B).
\end{equation}
On the other hand the right-hand side can obviously be written as
\begin{equation}\label{022}
\Phi_{p,\bw,k,l}(\h,B)= \sum_{m=0}^\infty m\cdot\#\{q\in\mathcal{P}_{p,\bw}^*(\h)\colon\mu_{q,k,l}(B)=m\}.
\end{equation}
Since $\mu_{q,k,l}(B)\leq n!/(n-k-2l)!$, the number of the summands  in the right-hand side is finite.

Consider a set $A_m\subset\mathbb{B}_{p,\bw}^{n+1}$  (which depends on $B$) consisting of all points $(a_0,\dots,a_n)\in\mathbb{B}_{p,\bw}^{n+1}$ such that
\[
\mu_{a_0+a_1x+\dots+a_nx^n,k,l}(B)=m.
\]
Then, by definition of a primitive polynomial,
$$
\#\{q\in\mathcal{P}_{p,\bw}(\h)\colon q\text{ is primitive and }\mu_{q,k,l}(B)=m\}=\lambda^*(QA_m).
$$
Hence it follows from the definition of a prime polynomial that
\begin{multline}\label{020}
\left|\#\{p\in\mathcal{P}_{p,\bw}^*(\h)\colon\mu_{q,k,l}(B)=m\} - \frac12 \lambda^*(QA_m)\right|
\\\leq\#\{q\in\mathcal{P}_{p,\bw}(\h)\colon q\text{ is reducible}\}.
\end{multline}
 Note that the factor $1/2$ arises  because prime polynomials have positive leading coefficients.
It is known (see~\cite{bW36}, or \cite{gK09}, \cite{aD14}) that
\begin{equation}\label{021}
\#\{q\in\mathcal{P}_{p,\bw}(\h)\colon q\text{ is reducible}\}=O\left(\h^n\log^{\chi_{n,0}}\h\right),\quad \h\to\infty,
\end{equation}
where $\chi_{n,0}$ is defined in~\eqref{2151}.
Note that for $n=2$ any reducible integer polynomial has only real roots. Thus~\eqref{020} and~\eqref{021} imply
$$
\left|\#\{p\in\mathcal{P}_{p,\bw}^*(\h)\colon\mu_{q,k,l}(B)=m\} - \frac12 \lambda^*(QA_m)\right|
=O\left(\h^n\log^{\chi_{n,l}}\h\right),\quad \h\to\infty.
$$
Applying this to~\eqref{022}, we obtain
\begin{equation}\label{245}
\Phi_{p,\bw,k,l}(\h,B)= \frac12 \sum_{m=0}^\infty m \lambda^*(\h A_m)+O\left(\h^n\log^{\chi_{n,l}}\h \right),\quad \h\to\infty.
\end{equation}

It follows from Lemma~\ref{lm-prim-count} that under the assumptions of Theorem~\ref{2235},
\begin{equation}\label{1445}
\lim_{\h\to\infty}\frac{\lambda^*(QA_m)}{\h^{n+1}}=\frac{\Vol_{n+1}(A_m)}{\zeta(n+1)}.
\end{equation}

To estimate the rate of convergence in~\eqref{1445} we need to use~\eqref{1318} which requires $A_m$ to be of Lipschitz class. This is established in the next lemma.
\begin{lemma}\label{lm-algebr}
Let $B\subset \R^k\times \C_+^l$ and $A_m$ be defined as above. If $B$ is of Lipschitz class $(M,L)$, 
then the boundary $\partial A_m$ is of Lipschitz class $(M_1,L_1)$ for some constants $M_1$, $L_1$ depending on $n, M, L, \bw$ only.
\end{lemma}
The proof of Lemma~\ref{lm-algebr} is postponed to \tc{the end of this section.} Now assuming Lemma~\ref{lm-algebr} we have from~\eqref{1318} that under the assumptions of Theorem~\ref{1224},
\begin{equation}\label{1446}
\left|\frac{\lambda^*(QA_m)}{\h^{n+1}}-\frac{\Vol_{n+1}(A_m)}{\zeta(n+1)}\right|\leq \frac{C}{\h},
\end{equation}
where $C$ depends on $n, M, L, \bw$ only. \tc{Note that since $n\geq2$, dimension of $A_m$ is at least 3, so there is no  factor $\log\h$ in the right-hand side of~\eqref{1446}}.

Thus having obtained~\eqref{245},~\eqref{1445}, and~\eqref{1446}, to prove Theorems~\ref{2235} and~\ref{1224}  we are left with the task of calculating of $\Vol_{n+1}(A_m)$. To this end, consider the random polynomial defined as
$$
\tilde G(z):=\sum_{k=0}^{n}\xi_k z^k,
$$
where the random vector $(\xi_0,\xi_1,\dots,\xi_n)$ is uniformly distributed over $\mathbb{B}_{p,\bw}^{n+1}$. By definition of $A_m$,
\begin{equation}\label{1430}
\P[\mu_{\tilde G,k,l}(B)=m]=\frac{\Vol_{n+1}(A_m)}{\Vol_{n+1}(\mathbb{B}_{p,\bw}^{n+1})}.
\end{equation}
The probability {on} the left-hand side is difficult to calculate \tc{due to dependence} of the coefficients of $\tilde G(z)$. However, the zeros of $\tilde G(z)$ do not change if we divide the polynomial by any
non-zero constant. By proper normalisation we can achieve independence of the coefficients.

\begin{lemma}\label{2252}
Let $p\in(0,\infty]$, and $\bw=(w_0, w_1,\dots,w_n)$ be a vector of positive weights. Assume that the random vector $(\xi_0,\xi_1,\dots,\xi_n)$ is uniformly distributed in $\mathbb{B}_{p,\bw}^{n+1}$
and that the random variables $\eta_0,\eta_1,\dots,\eta_n$ are  i.i.d. with
Lebesgue density
\begin{equation}\label{eq-dens}
f(t):=
\left\{
  \begin{array}{ll}
    \frac{1}{2\Gamma\left(1+\frac1p\right)}e^{-|t|^p}, & p<\infty, \\
    \frac{1}{2}\ind_{[-1,1]}(t), & p=\infty.
  \end{array}
\right.
\end{equation}
Then the polynomials $\tilde G(z)=\sum_{i=0}^n \xi_i z^i$ and
\begin{equation}\label{eq-G-ext}
G(z)=\sum_{i=0}^n \frac{\eta_i}{w_i} z^i
\end{equation}
have the same distribution of roots (both real and complex).
\end{lemma}
To prove the lemma, we need the following probabilistic representation of the uniform measure on $\mathbb{B}_p^{n+1}$.
\begin{theorem}[Barthe et al.]\label{thm-unif-ball}
Let $p>0$ and let $\eta_0,\eta_1,\dots,\eta_n$ be i.i.d. random variables with p.d.f. \eqref{eq-dens}.
Let $Z$ be an exponential random variable (i.e., the p.d.f. of $Z$ is $e^{-t},t\geq0$) independent of $\eta_0,\eta_1,\dots,\eta_n$. Then,
$$
\frac{(\eta_0,\eta_1,\dots,\eta_n)}{(\sum_{i=0}^{n}|\eta_i|^p+Z)^{1/p}}\eqdistr(\xi_0,\xi_1,\dots,\xi_n).
$$
\end{theorem}
\begin{proof}
  See~\cite{BGMN05}.
\end{proof}
Now we are ready to prove the lemma.
\begin{proof}[Proof of Lemma~\ref{2252}]
It is easy to check that the vector $(w_0 \xi_0, w_1 \xi_1, \dots, w_n \xi_n)$ is uniformly distributed in $\mathbb{B}_{p}^{n+1}$.
Therefore, from Theorem \ref{thm-unif-ball} we obtain at once
\[
\frac{(w_0^{-1}\eta_0,w_1^{-1}\eta_1,\dots,w_n^{-1}\eta_n)}{(\sum_{i=0}^{n}|\eta_i|^p+Z)^{1/p}}\eqdistr(\xi_0,\xi_1,\dots,\xi_n).
\]
Since dividing a polynomial by a non-zero constant does not affect its zeros,  the lemma is proved.
\end{proof}

Thus it readily follows that
$$
\P[\mu_{\tilde G,k,l}(B)=m]=\P[\mu_{G,k,l}(B)=m],
$$
where the random polynomial $G$ is defined in \eqref{eq-G-ext}. Combining this with~\eqref{1430}, we arrive at
\begin{align*}
\sum_{m=0}^\infty& m\Vol_{n+1}(A_m)=\Vol_{n+1}(\mathbb{B}_{p,\bw}^{n+1})\sum_{m=0}^\infty m\,\P[\mu_{G,k,l}(B)=m]
\\&=\Vol_{n+1}(\mathbb{B}_{p,\bw}^{n+1})\, \E [\mu_{G,k,l}(B)]=\Vol_{n+1}(\mathbb{B}_{p,\bw}^{n+1})\int_B\rho_{p,\bw,k,l}(\bx,\bz)\,\dd\bx\dd\bz,
\end{align*}
where the last relation follows from the properties of correlation functions (see, e.g.,~\cite{HKPV09}). Combining this with~\eqref{245},~\eqref{1445}, and~\eqref{1446} finishes the proof of
Theorems~\ref{2235},~\ref{1224} and \ref{1558}.

\subsection{Proof of Lemma~\ref{lm-algebr}}\label{2217}
\begin{claim*}
	The boundary of $A_m$ is contained in the union of four sets:
	\begin{enumerate}
		\item the boundary of $\mathbb{B}_{p,\bw}^{n+1}$;
		\item  the set
		\[
		A' = \left\{(a_0,\dots,a_n)\in\mathbb{B}_{p,\bw}^{n+1}:\mu_{a_0+a_1x+\dots+a_nx^n,k,l}(\partial B)>0\right\};
		\]
		\item the set $D$ consisting of the points $(a_0,\dots,a_n)$ such that  the polynomial $a_0+a_1x+\dots+a_nx^n$ has double real roots;
		\item the set $D'$ consisting of the points $(a_0,\dots,a_n)$ such that  the polynomial $a_0+a_1x+\dots+a_nx^n$ has double non-real roots.
	\end{enumerate}
\end{claim*}
\begin{proof}[Proof of Claim]
Suppose that a point $(a_0,\dots,a_n)$ does not belong to any of  Sets~(i)--(iv). The task is to show that $(a_0,\dots,a_n)\not\in \tv{\partial A_m}$. Let $B'\subset\C^{k+l}$ denote a set of all $(k+l)$-tuples of different (real or complex) roots of 
\begin{align*}
    g(x):=a_0+a_1x+\dots+a_nx^n.
\end{align*}
Since $B'$ is finite and $(a_0,\dots,a_n)\not\in A'$, we have
\begin{align}\label{2101}
    \varepsilon_1:=\dist(\partial B,B')>0.
\end{align}
Denoting by $z_1,\dots,z_n$ the roots of $g$, we have that
\begin{align}\label{3}
    \varepsilon_2:=\min_{i\ne j}|z_i- z_j|>0.
\end{align}
Consider $\delta>0$ and a polynomial
\begin{align*}
    h(x):=b_0+b_1x+\dots+b_nx^n
\end{align*}
such that 
\begin{align*}
    \tv{|a_j-b_j|<\delta}\quad\text{for}\quad j=0,\dots,n.
\end{align*}
If $\delta$ is sufficiently small, then 
\begin{align}\label{1930}
    (b_0,\dots,b_n)\in\mathbb{B}_{p,\bw}^{n+1}\quad\text{if and only if}\quad(a_0,\dots,a_n)\in\mathbb{B}_{p,\bw}^{n+1},
\end{align}
\tv{which means that} $(b_0,\dots,b_n)$ does not belong to Set (i).
Let
\begin{align*}
    R:=\max_{i=1,\dots,n}|z_i|.
\end{align*}
If $z_i$ is a root of $g$, then
\begin{align*}
    |h(z_i)|=|g(z_i)-h(z_i)|\leq\sum_{j=0}^n|a_j-b_j|R^j\leq(n+1)\delta(R^n+1).
\end{align*}
On the other hand, 
\begin{align*}
    h(z_i)=b_n(z_i-z'_1)\dots(z_i-z'_n),
\end{align*}
where $z'_1,\dots,z'_n$ are the roots of $h$. Assuming that $\delta$ is so small that $|b_n|\geq |a_n|/2$, we arrive at
\begin{align*}
    \bigl|(z_i-z'_1)\dots(z_i-z'_n)\bigr|\leq2(n+1)a_n^{-1}\delta(R^n+1).
\end{align*}
Therefore if $\delta$ is sufficiently small, then there exists an index $j$ such that
\begin{align}\label{2038}
    |z_i-z'_j|<\frac{\varepsilon_2} 2\quad\text{and}\quad  |z_i-z'_j|<\frac{\varepsilon_1}{\sqrt{k+l}}.
\end{align}
Thus, taking into account~\eqref{3}, we obtain that the roots of $g$ and $h$ split into $n$ pairs all satisfying~\eqref{2038} and the roots $z'_1,\dots,z'_n$ are pairwise different. Now consider some indices $1\leq i_1<\dots<i_{k+l}\leq n$. It follows from the second part of~\eqref{2038} that
\begin{align*}
    \bigl\|(z_{i_1},\dots,z_{i_{k+l}})-(z'_{i_1},\dots,z'_{i_{k+l}})\bigr\|<\varepsilon_1,
\end{align*}
which together with~\eqref{2101} implies that $(z'_{i_1},\dots,z'_{i_{k+l}})\not\in\partial B$
and
\begin{align*}
    (z'_{i_1},\dots,z'_{i_{k+l}})\in B\quad\text{if and only if}\quad \tv{(z_{i_1},\dots,z_{i_{k+l}})}\in B.
\end{align*}
This together with~\eqref{1930} means that all \tv{points} $(b_0,\dots,b_n)$ from some neighborhood of $(a_0,\dots,a_n)$  satisfy
\begin{align*}
    (b_0,\dots,b_n)\in \tv{A_m}\quad\text{if and only if}\quad(a_0,\dots,a_n)\in \tv{A_m},
\end{align*}
so $(a_0,\dots,a_n)\not\in\tv{\partial A_m}$, and the claim follows.
\end{proof}	
\par
It follows from the claim that in \tv{order} to prove Lemma~\ref{lm-algebr} it is enough to show that Sets~(i)--(iv) are of Lipschitz class.

	{\bf (i) The boundary of $\mathbb{B}_{p,\bw}^{n+1}$.}	 {For $p\ge 1$ the set $\mathbb{B}_{p,\bw}^{n+1}$ is a convex body. Therefore, according to~\cite[Theorem~2.6]{mW12} its boundary $\partial\mathbb{B}_{p,\bw}^{n+1}$ belongs to the Lipschitz class $(1,L_0)$, where $L_0$ depends  on $n,p,\bw$ only.
In the case $0<p<1$ the set $\partial\mathbb{B}_{p,\bw}^{n+1}$ consists of $2^{n+1}$ congruent concave pieces.
Every  piece can be embedded in the boundary of a convex body congruent to
\[
\left\{(a_0,\dots,a_n)\in \mathbb{B}_{1,\bw}^{n+1}\setminus \mathbb{B}_{p,\bw}^{n+1}: \min_{0\le j\le n} a_j \ge 0\right\}.
\]
Applying~\cite[Theorem~2.6]{mW12} to the latter set we see that $\partial\mathbb{B}_{p,\bw}^{n+1}$ is of Lipschitz class $(2^{n+1},L_0)$ for $0<p<1$.}
	
	{\bf (ii) The set $A'$.} 
	Consider binary multi-indices $\epsilon\in\{-1,1\}^k,\delta\in\{-1,1\}^l$. 	We have
	\[
	A'\subset\bigcup_{(\epsilon,\delta)\in\{-1,1\}^{k+l}} A'_{\epsilon,\delta},
	\]
where $A'_{\epsilon,\delta}$ is a set of points $(a_0,\dots,a_n)\in\mathbb{B}_{p,\bw}^{n+1}$ such that there exists a $(k+l)$-tuple
	\begin{equation}\label{1405}
	(x_1,\dots,x_k,z_1,\dots,z_{l})\in\partial B
	\end{equation} 
	of distinct zeros of $a_0+a_1x+\dots+a_nx^n$ with
	\begin{equation}\label{1436}
	\begin{cases}
	|x_i|\leq1, & \epsilon_i=1,\\
	|x_i|>1, & \epsilon_i=-1
	\end{cases}
	\quad\mathrm{and}
	\quad
	\begin{cases}
	|z_j|\leq1, & \delta_j=1,\\
	|z_j|>1, & \delta_j=-1
	\end{cases}
	\end{equation}
	for $i=1,\dots,k,j=1,\dots,l$.
	
	Let us fix some $(\epsilon,\delta)\in\{-1,1\}^{k+l}$, $(a_0,\dots,a_n)\in A'_{\epsilon,\delta}$, and the corresponding $(k+l)$-tuple~\eqref{1405}.
	
	Since the components of~\eqref{1405} are different roots of $a_0+a_1x+\dots+a_nx^n$, there exists a unique polynomial with real coefficients  $\sum_{r=0}^{n-k-2l} b_r x^r$ such that
	\begin{align}\label{eq-B-param}
	\sum_{i=0}^n a_i x^i=\sum_{r=0}^{n-k-2l} b_r x^r
	&\prod_{i:\epsilon_i=1} (x-x_i) \prod_{i:\epsilon_i=-1}(x_i^{-1}x-1)\\
	\times&\prod_{j:\delta_j=1} (x-z_j)(x-\bar{z}_j)  \prod_{j:\delta_j=-1}(z_j^{-1}x-1)(\bar{z}_j^{-1}x-1). \notag
	\end{align}
	Thus,
	\begin{equation}\label{1452b}
	a_i=p_i(x_1^{\epsilon_1},\dots,x_k^{\epsilon_k},z_1^{\delta_1},\dots,z_{l}^{\delta_l},\bar z_1^{\delta_1},\dots,\bar z_{l}^{\delta_l},b_0,\dots,b_{n-k-2l})
	\end{equation}
	for some polynomials $p_0,\dots,p_n$.

	Since $\partial B$ is of Lipschitz class, there exists a Lipschitz map 
	\[
	\phi=(\phi_1,\dots,\phi_{k+l}) : [0, 1]^{k+2l-1} \to\R^k\times \C_+^l
	\]
	(from a fixed finite collection of Lipschitz maps depending on $B$ only) such that 
	\[
	(x_1,\dots,x_k,z_1,\dots,z_{l})\in\phi( [0, 1]^{k+2l-1}).
	\]
	Thus for some $\bt_0\in  [0, 1]^{k+2l-1}$ and $i=1,\dots,k$, $j=1,\dots,l$ we have
	\begin{equation}\label{1452}
	x_i=\phi_i(\bt_0), \qquad z_j=\phi_{k+j}(\bt_0).
	\end{equation}
	For $i=1,\dots,k$, $j=1,\dots,l$, let
	\[
	\varphi_{i,\epsilon_i}:=
	\begin{cases}
	\phi_i, & \epsilon_i=1,\\
	\frac{\phi_i}{\max(1,|\phi_i|^2)}, & \epsilon_i=-1,
	\end{cases}
	\quad\text{and}\quad
	\varphi_{k+j,\delta_j}:=
	\begin{cases}
	\phi_{k+j}, & \delta_j=1,\\
	\frac{\phi_{k+j}}{\max(1,|\phi_{k+j}|^2)}, & \delta_j=-1.
	\end{cases}
	\]
	Obviously, $\varphi_{i,\epsilon_i}$, $\bar\varphi_{i,\epsilon_i}$, $\varphi_{k+j,\delta_j}$, $\bar\varphi_{k+j,\delta_j}$ are Lipschitz, too. Taking into account~\eqref{1436},~~\eqref{1452b}, and~\eqref{1452}, we have
%	\begin{equation}\label{1504}
%	a_i=p_i(\varphi_{1,\epsilon_1}(t_0),\dots,\varphi_{k+l,\delta_l}(t_0),\bar\varphi_{k+1,\delta_1}(t_0),\dots,\bar\varphi_{k+l,\delta_l}(t_0),b_0,\dots,b_{n-k-2l}).
%	\end{equation}
	\begin{align}\label{1504}
	a_i=p_i(\varphi_{1,\epsilon_1}(\bt_0),\dots,\varphi_{k,\epsilon_k}(\bt_0),\, & \varphi_{k+1,\delta_1}(\bt_0),\dots, \varphi_{k+l,\delta_l}(\bt_0),\\
	&\bar\varphi_{k+1,\delta_1}(\bt_0),\dots,\bar\varphi_{k+l,\delta_l}(\bt_0),
	b_0,\dots,b_{n-k-2l}). \notag
	\end{align}
	Let us assume for the moment that there exists a constant $C>0$ depending on $B$ only such that
	\begin{equation}\label{1526}
	|b_0|,|b_1|,\dots, |b_{n-k-2l}|\leq C.
	\end{equation}
	Letting
	\[
\tv{	\tilde b_i:=\frac{b_i+C}{2C}}
	\]
	and  redefining the polynomials $p_0,\dots,p_n$ accordingly, we obtain
%	\[
%	a_i=\tilde p_i(\varphi_{1,\epsilon_1}(t_0),\dots,\varphi_{k+l,\delta_l}(t_0),\bar\varphi_{k+1,\delta_1}(t_0),\dots,\bar\varphi_{k+l,\delta_l}(t_0),\tilde b_0,\dots,\tilde b_{n-k-2l}).
%	\]
	\begin{align*}
	a_i=\tilde p_i(\varphi_{1,\epsilon_1}(\bt_0),\dots,\varphi_{k,\epsilon_k}(\bt_0),\, & \varphi_{k+1,\delta_1}(\bt_0),\dots, \varphi_{k+l,\delta_l}(\bt_0),\\
	&\bar\varphi_{k+1,\delta_1}(\bt_0),\dots,\bar\varphi_{k+l,\delta_l}(\bt_0),
	\tilde b_0,\dots,\tilde b_{n-k-2l}),
	\end{align*}
	where
	\[
	(\bt_0,\tilde b_0,\dots,\tilde b_{n-k-2l})\in[0,1]^n.
	\]
	Since polynomials are Lipschitz on compact sets and the composition of Lipschitz functions is Lipschitz, assuming~\eqref{1526} it follows from~\eqref{1504} that $A'$ is of Lipschitz class. Now let us show that~\eqref{1526} holds. To this end, we first recall the definition of the Mahler measure and its basic properties.
	
	For a polynomial $q(z)=a_n(z-z_1)\dots(z-z_n)$ its Mahler measure is defined as
	\[
	M[q]:=|a_n|\prod_{i=1}^{n}\max(1,|z_i|).
	\]
	It is known (see, e.g.,~\cite[Theorem~4.2.1]{vP04}) that the naive height can be estimated via the Mahler measure as follows:
	\begin{equation}\label{1810}
	\frac{M[q]}{\sqrt{n+1}}\leq H[q]\leq 2^{n-1}M[q].
	\end{equation}
	Also, it is easily seen that for polynomials $q_1,q_2$ one has $M[q_1q_2]=M[q_1]M[q_2]$. Thus the polynomials $\sum_{i=0}^n a_i x^i$ and $\sum_{r=0}^{n-k-2l} b_r x^r$ (see~\eqref{eq-B-param}) have the same Mahler measure. We have
	\begin{align*}
	H\!\left[\sum\nolimits_{r=0}^{n-k-2l} b_r x^r\right]&\leq 2^{n-k-2l-1}M\!\left[\sum\nolimits_{r=0}^{n-k-2l} b_r x^r\right]=2^{n-k-2l-1}M\!\left[\sum\nolimits_{i=0}^n a_i x^i\right]\\
	&\leq2^{n-k-2l-1}\sqrt{n+1}\,H\!\left[\sum\nolimits_{i=0}^n a_i x^i\right]\leq2^{n-k-2l-1}\sqrt{n+1},
	\end{align*}
	and~\eqref{1526} follows.
	
	{\bf (iii) The set $D$.}
	\tc{Let
	\begin{align*}
	    W:=[-1,1]\times\big[-\max_{j=0,\dots,n} w_j^{-1}, \max_{j=0,\dots,n} w_j^{-1}\big]^{n-1}.
	\end{align*}}
	Consider \tc{a map} 
	\[
	\phi=(\phi_0,\dots,\phi_{n}) : \tc{W} \to\R^{n+1}
	\]
	defined as 
	\begin{align*}
	\phi_0(\tc{x},t_2,\dots,t_{n}) &= \sum_{j=2}^{n} (j-1) \tc{t_jx}^{j}, \quad
	\phi_1(\tc{x},t_2,\dots,t_{n}) = - \sum_{j=2}^{n} j \tc{t_jx}^{j-1}, \\
	\phi_j(\tc{x},t_2,\dots,t_{n}) &= \tv{t_j},\quad 2\le j\le n.
	\end{align*}
	Since $\phi$ \tc{is} continuously differentiable in a compact, \tc{it satisfies} \tb{a} Lipschitz condition with some constant  which depends  on $n$ and \tc{$\max_{j=0,\dots,n}w_j^{-1}$} only.
	
	Now suppose that $(a_0,\dots,a_n)\in D$. Then $a_0+a_1x+\dots+a_nx^n$ has a multiple real root, say $x_0$, which implies
	\begin{equation}\label{1441}
	\sum_{j=0}^n a_j x_0^j = 0, \qquad
	\sum_{j=0}^n ja_j x_0^{j-1} = 0,
	\end{equation}
	or, equivalently,
	\begin{equation*}
	a_0 = \sum_{j=2}^n (j-1) a_jx_0^j, \qquad
	a_1 = - \sum_{j=2}^n j a_jx_0^{j-1}.
	\end{equation*}
	Using these equations and definition of $\phi$ we arrive at
	\begin{equation}\label{1457}
	a_j=\phi_j(x_0,\tc{a_2},\dots,\tc{a_n}),\quad 0\le j\le n.
	\end{equation}
	\tc{
	Moreover, it is straightforward that if $x_0\ne0$, then $1/x_0$ is a multiple root of the reflected polynomial $a_n+a_{n-1}x+\dots+a_0x^n$, \tb{hence}  applying the above reasoning gives
	\begin{equation*}
	    a_{n-j}=\phi_j(x_0^{-1},a_{n-2},\dots,a_{0}),\quad 0\le j\le n,
	\end{equation*}
	or, equivalently,
	\begin{equation}\label{1458}
	    a_{j}=\phi_{n-j}(x_0^{-1},a_{n-2},\dots,a_{0}),\quad 0\le j\le n.
	\end{equation}
	}

	Since $(a_0,\dots,a_n)\in\mathbb{B}_{p,\bw}^{n+1}\subseteq \mathbb{B}_{\infty,\bw}^{n+1}$, we have that for $|x_0|\leq1$ the argument of the right-hand side of~\eqref{1457} belongs to \tc{$W$}, while  for $|x_0|\geq1$ the argument of the right-hand side of~\eqref{1458} belongs to \tc{$W$}, \tc{\tb{hence}  we have  $(a_0,\dots,a_n)\in\phi( W)\cup\tc{\phi'}( W)$, where $\phi':=(\phi_n,\dots,\phi_0)$. Thus it follows from Remark~\ref{2006} }that $D$ is of  Lipschitz class.

	\emph{\bf (iv) The set $D'$.}	Let
	\begin{align*}
	    W:=[0,\pi]\times[0,1]\times\big[-\max_{j=0,\dots,n} w_j^{-1}, \max_{j=0,\dots,n} w_j^{-1}\big]^{n-2}.
	\end{align*}
	Consider a map
	\[
	\phi=(\phi_0,\dots,\phi_{n}) : W \to\R^{n+1}
	\]
	defined as 
	\begin{align*}
	\phi_0(\alpha,r,t_3,\dots,t_{n}) &= \sum_{j=3}^n t_jr^{j}\left(-\frac j2\frac{\sin[ (j-1)\alpha]}{\sin\alpha}+ j\cos\alpha\cos[(j-1)\alpha]
	-\cos\alpha\right),
	\\\phi_1(\alpha,r,t_3,\dots,t_{n})& =\sum_{j=3}^n j t_jr^{j}\left( \cos\alpha\frac{\sin [(j-1)\alpha]}{\sin\alpha}- \cos [(j-1)\alpha]\right), 
	\\\phi_2(\alpha,r,t_3,\dots,t_{n})& =-\sum_{j=3}^n jt_jr^{j}\frac{\sin[ (j-1)\alpha]}{\sin[2\alpha]}, 
	\\	\phi_j(\alpha,r,t_3,\dots,t_{n})& = t_j,\quad 3\le j\le n.
	\end{align*} 
	Again, $\phi$ is continuously differentiable in a compact, so it satisfies the Lipschitz condition with some constant  which depends  on $n$ and $\max_{j=0,\dots,n}w_j^{-1}$ only.
	
	Now suppose that $(a_0,\dots,a_n)\in D'$. Then $a_0+a_1x+\dots+a_nx^n$ has a multiple non-real root, say $z_0=r_0(\cos\alpha_0+\mathbf{i} \sin\alpha_0)$, where $r_0>0, \alpha_0\in(0,\pi)$, which implies
	\begin{equation*}
	\sum_{j=0}^n a_j z_0^j = 0, \qquad
	\sum_{j=1}^n ja_j z_0^{j-1} = 0,
	\end{equation*}
	or, equivalently,
	\begin{align}\label{2247}
	    \sum_{j=0}^n a_j r_0^j\cos [j\alpha_0] = 0, \qquad	&\sum_{j=1}^n ja_j r_0^{j-1}\cos[(j-1)\alpha_0] = 0,
	    \\\sum_{j=1}^n a_j r_0^j\sin[j\alpha_0] = 0, \qquad	&\sum_{j=2}^n ja_j r_0^{j-1}\sin[(j-1)\alpha_0] = 0.\notag
	\end{align}
	It consistently follows from the fourth,  second, and first equalities in~\eqref{2247} that
	\begin{align*}
	    a_2&=-\frac12\sum_{j=3}^n ja_j r_0^{j-2}\frac{\sin [(j-1)\alpha_0]}{\sin\alpha_0},
	    \\a_1&=-2a_2r_0\cos[\alpha_0]-\sum_{j=3}^n ja_j r_0^{j-1}\cos [(j-1)\alpha_0]
	    \\&=\sum_{j=3}^n ja_j r_0^{j-1}\cos[\alpha_0]\frac{\sin [(j-1)\alpha_0]}{\sin\alpha_0}-\sum_{j=3}^n ja_j r_0^{j-1}\cos [(j-1)\alpha_0],
	    \\a_0&=-a_1r_0\cos\alpha_0-a_2r_0^2\cos[2\alpha_0]-\sum_{j=3}^n a_j r_0^j\cos [j\alpha_0]
	    \\&=-\sum_{j=3}^n ja_j r_0^{j}\cos^2[\alpha_0]\frac{\sin [(j-1)\alpha_0]}{\sin\alpha_0}+\sum_{j=3}^n ja_j r_0^{j}\cos\alpha_0\cos[(j-1)\alpha_0]
	    \\&\;\;\;\;+\frac12\sum_{j=3}^n ja_j r_0^{j}\cos[2\alpha_0]\frac{\sin [(j-1)\alpha_0]}{\sin\alpha_0}-\sum_{j=3}^n a_j r_0^j\cos [j\alpha_0]
	    \\&=-\frac12\sum_{j=3}^n ja_j r_0^{j}\frac{\sin [(j-1)\alpha_0]}{\sin\alpha_0}+\sum_{j=3}^n ja_j r_0^{j}\cos\alpha_0\cos [(j-1)\alpha_0]-\sum_{j=3}^n a_j r_0^j\cos [j\alpha_0],
	\end{align*}
	where in the last equation we used the identity $\cos[2\alpha_0]=2\cos^2\alpha_0-1$.
	Using these equations and the definition of $\phi$ we arrive at
	\begin{equation}\label{12}
	a_j=\phi_j(\alpha_0,r_0,a_3,\dots,a_n),\quad 0\le j\le n.
	\end{equation}
    Again, $1/\conj{z_0}$ is a multiple root of the reflected polynomial $a_n+a_{n-1}x+\dots+a_0x^n$, so  applying the above reasoning gives
    \begin{equation*}
	a_{n-j}=\phi_j(\alpha_0,r_0,a_{n-3},\dots,a_0),\quad 0\le j\le n,
	\end{equation*}
	or, equivalently,
	\begin{equation}\label{8}
	a_{j}=\phi_{n-j}(\alpha_0,r_0,a_{n-3},\dots,a_0),\quad 0\le j\le n.
	\end{equation}
	Since $(a_0,\dots,a_n)\in\mathbb{B}_{p,\bw}^{n+1}\subseteq \mathbb{B}_{\infty,\bw}^{n+1}$, we have that for $r_0\leq1$ the argument of the right-hand side of~\eqref{12} belongs to $W$, while  for $r_0\geq1$ the argument of the right-hand side of~\eqref{8} belongs to $W$, so we have  $(a_0,\dots,a_n)\in\phi( W)\cup\phi'( W)$, where $\phi':=(\phi_n,\dots,\phi_0)$. Thus it follows from Remark~\ref{2006} that $D'$ is of  Lipschitz class.

\section{Proof of Theorem~\ref{thm-corr-rel}}\label{1245}\label{sc-proof-rho}

\subsection{Preliminaries}

Suppose that $x_1,\dots,x_k\in\R$ and $z_1,\dots,z_l\in\C_+$ are different zeros of the random polynomial $G$ defined in \eqref{1914}.
It means that
\begin{equation}\label{2018}
\begin{pmatrix}
1 & x_1 & \dots & x_1^{n} \\
\vdots & \vdots & \ddots & \vdots \\
1 & x_k & \dots & x_k^{n} \\
1 & \Re z_1 & \dots & \Re z_1^{n} \\
0 & \Im z_1 & \dots & \Im z_1^{n} \\
\vdots & \vdots & \ddots & \vdots \\
1 & \Re z_l & \dots & \Re z_l^{n} \\
0 & \Im z_l & \dots & \Im z_l^{n} \\
\end{pmatrix}
\begin{pmatrix}
\xi_0\\
\vdots\\
\xi_n
\end{pmatrix}
=\mathbf{0}.
\end{equation}

Denote by $V(\bx,\bz)$ {the real Vandermonde type} matrix
$$
V(\bx,\bz):=
\begin{pmatrix}
1 & x_1 & \dots & x_1^{k+2l-1} \\
\vdots & \vdots & \ddots & \vdots \\
1 & x_k & \dots & x_k^{k+2l-1} \\
1 & \Re z_1 & \dots & \Re z_1^{k+2l-1} \\
0 & \Im z_1 & \dots & \Im z_1^{k+2l-1} \\
\vdots & \vdots & \ddots & \vdots \\
1 & \Re z_l & \dots & \Re z_l^{k+2l-1} \\
0 & \Im z_l & \dots & \Im z_l^{k+2l-1} \\
\end{pmatrix}.
$$
Then,~\eqref{2018} is equivalent to
\begin{equation}\label{1400}
\begin{pmatrix}
\sum_{j=k+2l}^n\xi_jx_1^j\\
\vdots\\
\sum_{j=k+2l}^n\xi_jx_k^j\\
\Re\sum_{j=k+2l}^n\xi_jz_1^j\\
\Im\sum_{j=k+2l}^n\xi_jz_1^j\\
\vdots\\
\Re\sum_{j=k+2l}^n\xi_jz_l^j\\
\Im\sum_{j=k+2l}^n\xi_jz_l^j\\
\end{pmatrix}
=-V(\bx,\bz)
\begin{pmatrix}
\xi_{0}\\
\vdots\\
\xi_{k+2l-1}
\end{pmatrix}.
\end{equation}
It is easy to check that $V(\bx,\bz)$ satisfies
\begin{equation}\label{1931}
|\det V(\bx,\bz)|=2^{-l}\rv_{k+2l}(\bx,\bz,\bar\bz),
\end{equation}
where $\rv_{k+2l}$ is defined in~\eqref{1330}.

Consider a random function $\be=(\eta_0,\dots,\eta_{k+2l-1})^T:\R^{k}\times\C^l\to\R^{k+2l}$ defined as
\begin{equation}\label{2327}
  \be(\bx,\bz):=
  -V^{-1}(\bx,\bz)
\begin{pmatrix}
\sum_{j=k+2l}^n\xi_jx_1^j\\
\vdots\\
\sum_{j=k+2l}^n\xi_jx_k^j\\
\Re\sum_{j=k+2l}^n\xi_jz_1^j\\
\Im\sum_{j=k+2l}^n\xi_jz_1^j\\
\vdots\\
\Re\sum_{j=k+2l}^n\xi_jz_l^j\\
\Im\sum_{j=k+2l}^n\xi_jz_l^j\\
\end{pmatrix}.
\end{equation}
%\begin{equation}\label{2327}
%  \mathbf{\eta}(\bx,\bz)=
%  -V^{-1}(\bx,\bz)
%\begin{pmatrix}
%\sum_{j=k+2l}^n\xi_jx_1^j\\
%\vdots\\
%\sum_{j=k+2l}^n\xi_jx_k^j\\
%\sum_{j=k+2l}^n\xi_jz_1^j\\
%\vdots\\
%\sum_{j=k+2l}^n\xi_jz_l^j\\
%\end{pmatrix}.
%\end{equation}
It follows from~\eqref{1400} that~\eqref{2018} is equivalent to
\begin{equation}\label{1404}
\be(\bx,\bz)=
 \begin{pmatrix}
\xi_0\\
\vdots\\
\xi_{k+2l-1}
\end{pmatrix}.
\end{equation}

Consider a random function $\varphi:\R^k\times\C^{l}\to\R$ defined as
\begin{equation}\label{eq-phi-def}
\begin{aligned}
\varphi(\bx,\bz):=\frac1{\rv_{k+2l}(\bx,\bz)}&\prod_{i=1}^{k}\Bigg|\sum_{j=0}^{k+2l-1}j\eta_j(\bx,\bz)x_i^{j-1} + \sum_{j=k+2l}^{n}j\xi_j x_i^{j-1}\Bigg|\\
&\times\prod_{i=1}^{l}\Bigg|\sum_{j=0}^{k+2l-1}j\eta_j(\bx,\bz)z_i^{j-1} + \sum_{j=k+2l}^{n}j\xi_j z_i^{j-1}\Bigg|^2.
\end{aligned}
\end{equation}

\begin{lemma}
%  Let $m=k+2l$.
  For all $(\bx,\bz)\in\R^k\times\C^{l}$,
\begin{equation}\label{1835}
 \E\left[\varphi(\bx,\bz)\prod_{i=0}^{k+2l-1}f_i(\eta_i(\bx,\bz))\right] =  \rho_{k+2l}(\bx,\bz,\bar\bz).
\end{equation}
\end{lemma}
\begin{proof}
The idea of the proof goes back to~\cite[pp. 58--59]{dK12-Trudy} (see also~\cite[Lemmas~2.5, 2.6]{dK14}).

By definition of the expected value,
\begin{align}\label{1850}
\E\left[\varphi(\bx,\bz)\prod_{i=0}^{k+2l-1}f_i(\eta_i(\bx,\bz))\right]=&\frac1{\rv_{k+2l}(\bx,\bz)}\\
\times&\int_{\R^{n-k-2l+1}}\prod_{i=1}^{k}\Bigg|\sum_{j=0}^{k+2l-1}j r_j(\bx,\bz,\bs)x_i^{j-1} + \sum_{j=k+2l}^{n}j s_j x_i^{j-1}\Bigg|\notag\\
\times&\prod_{i=1}^{l}\Bigg|\sum_{j=0}^{k+2l-1}j r_j(\bx,\bz,\bs)z_i^{j-1} + \sum_{j=k+2l}^{n}j s_j z_i^{j-1}\Bigg|^2\notag\\
\times&\prod_{i=0}^{k+2l-1}f_i(r_i(\bx,\bz,\bs))\prod_{i=k+2l}^{n}f_i(s_i)\dd s_{k+2l}\dots \dd s_n,\notag
\end{align}
where the functions $r_0,\dots,r_{k+2l-1}$ are defined by
\begin{equation}\label{1800}
\begin{pmatrix}
r_0(\bx,\bz,\bs)\\
\vdots\\
r_{k+2l-1}(\bx,\bz,\bs)\\
\end{pmatrix}
:=
  -V^{-1}(\bx,\bz)
\begin{pmatrix}
\sum_{j=k+2l}^ns_jx_1^j\\
\vdots\\
\sum_{j=k+2l}^ns_jx_k^j\\
\Re\sum_{j=k+2l}^ns_jz_1^j\\
\Im\sum_{j=k+2l}^ns_jz_1^j\\
\vdots\\
\Re\sum_{j=k+2l}^ns_jz_l^j\\
\Im\sum_{j=k+2l}^ns_jz_l^j\\
\end{pmatrix}
\end{equation}
and $\bs:=(s_{k+2l},\dots,s_n)$.

Now we {perform} the following change of variables:
\begin{equation}\label{1819}
s_i=\sum_{j=0}^{n-k-2l}(-1)^{k+2l-i+j}\sigma_{k+2l-i+j}(\bx,\bz,\bar{\bz})t_j,\quad i=k+2l,\dots,n.
\end{equation}
where $\sigma_i$'s are defined in~\eqref{1738} and we write $\sigma_i:=0$ for $i<0$. The Jacobian is a lower triangle matrix with {ones on} the diagonal, hence the determinant is 1. {This variable change is suggested by the fact that}
$x_1,\dots,x_k,z_1,\bar{z_1}\dots,z_l,\bar{z_l}$ are zeros of the polynomial
$$
g(z):=r_0+r_1z+\dots+r_{k+2l-1}z^{k+2l-1}+s_{k+2l}z^{k+2l}+\dots+s_nz^n,
$$
see~\eqref{1800}. Thus for some $t'_0,\dots,t'_{n-k-2l}$ we have
\begin{align}\label{1817}
g(z)&=\prod_{j=1}^{k}(z-x_j)\prod_{j=1}^{l}(z-z_j)(z-\bar{z_j})\left(\sum_{j=0}^{n-k-2l} t'_j z^j \right)\\
&=\left(\sum_{j=0}^{k+2l} (-1)^{k+2l-j} \sigma_{k+2l-j}(\bx,\bz,\bar{\bz}) z^j\right)\left(\sum_{j=0}^{n-k-2l} t'_j z^j \right).\notag
\end{align}
Comparing the coefficients of the polynomials from the left-hand  and right-hand sides and recalling~\eqref{1819} we obtain that $t'_i=t_i$ for $i=0,\dots,n-k-2l$ and
\begin{equation}\label{1846}
r_i(\bx,\bz,\bs)=\sum_{j=0}^{n-k-2l}(-1)^{k+2l-i+j}\sigma_{k+2l-i+j}(\bx,\bz,\bar{\bz})t_j,\quad i=0,\dots,k+2l-1.
\end{equation}
If we differentiate the first equation in~\eqref{1817} at points $x_i,z_i$, and $\bar{z_i}$, we get
\begin{align}\label{1827}
  \sum_{j=0}^{k+2l-1}j r_jx_i^{j-1} + \sum_{j=k+2l}^{n}j s_j x_i^{j-1} &= \prod_{j\ne i}(x_i-x_j)\prod_{j=1}^{l}(x_i-z_j)(x_i-\bar{z_j})\left(\sum_{j=0}^{n-k-2l} t'_j x_i^j \right), \\
\sum_{j=0}^{k+2l-1}j r_jz_i^{j-1} + \sum_{j=k+2l}^{n}j s_j z_i^{j-1} &= \prod_{j=1}^{k}(z_i-x_j)\prod_{j\ne i}(z_i-z_j)(z_i-\bar{z_j})\left(\sum_{j=0}^{n-k-2l} t'_j z_i^j \right),\notag \\
\sum_{j=0}^{k+2l-1}j r_j\bar{z_i}^{j-1} + \sum_{j=k+2l}^{n}j s_j \bar{z_i}^{j-1} &= \prod_{j=1}^{k}(\bar{z_i}-x_j)\prod_{j\ne i}(\bar{z_i}-z_j)(\bar{z_i}-\bar{z_j})\left(\sum_{j=0}^{n-k-2l} t'_j \bar{z_i}^j \right).\notag
\end{align}
Substituting~\eqref{1819}, \eqref{1846}, and~\eqref{1827} in~\eqref{1850} completes the proof of the lemma.

\end{proof}

Denote by $J_{\be}(\bx,\bz)$ the \emph{real} Jacobian matrix of $\be$ at point  $(\bx,\bz)$:
$$
J_{\be}=
\begin{pmatrix}
\frac{\partial \eta_0 }{\partial x_1} & \dots & \frac{\partial \eta_0 }{\partial x_k} & \frac{\partial \eta_0 }{\partial \Re z_1} &  \frac{\partial \eta_0 }{\partial \Im z_1}&\dots & \frac{\partial \eta_0 }{\partial \Re
z_l}& \frac{\partial \eta_0 }{\partial \Im z_l}\\
\vdots & \ddots & \vdots & \vdots &  \vdots&\ddots & \vdots&\vdots\\
\frac{\partial \eta_{k+2l-1} }{\partial x_1} & \dots & \frac{\partial \eta_{k+2l-1} }{\partial x_k} & \frac{\partial \eta_{k+2l-1} }{\partial \Re z_1} &  \frac{\partial \eta_{k+2l-1} }{\partial \Im z_1}&\dots &
\frac{\partial \eta_{k+2l-1} }{\partial \Re z_l}& \frac{\partial \eta_{k+2l-1} }{\partial \Im z_l}\\
\end{pmatrix}
.
$$

\begin{lemma}\label{1200}
  For all $(\bx,\bz)\in\R^k\times\C^{l},$
\begin{equation}\label{1233}
|\det J_{\be}(\bx,\bz)|
%=v(\bx,\bz)
% \prod_{i=1}^{k}\Bigg|\sum_{j=0}^{k+2l-1}j\eta_j(\bx)x_i^{j-1} + \sum_{j=k+2l}^{n}j\xi_j x_i^{j-1}\Bigg|\\
% \times\prod_{i=1}^{k}\Bigg|\sum_{j=0}^{k+2l-1}j\eta_j(\bz)z_i^{j-1} + \sum_{j=k+2l}^{n}j\xi_j z_i^{j-1}\Bigg|.
=2^l\varphi(\bx,\bz),
\end{equation}
where $\varphi(\bx,\bz)$ is defined in \eqref{eq-phi-def}.
\end{lemma}
\begin{proof}
Differentiating
\[
V(\bx,\bz)\be(\bx,\bz)=
  -
\begin{pmatrix}
x_1^{k+2l} & x_1^{k+2l+1} & \dots & x_1^{n} \\
\vdots & \vdots & \ddots & \vdots \\
x_k^{k+2l} & x_k^{k+2l+1} & \dots & x_k^{n} \\
\Re z_1^{k+2l} & \Re z_1^{k+2l+1} & \dots & \Re z_1^{n} \\
\Im z_1^{k+2l} & \Im z_1^{k+2l+1} & \dots & \Im z_1^{n} \\
\vdots & \vdots & \ddots & \vdots \\
\Re z_l^{k+2l} & \Re z_l^{k+2l+1} & \dots & \Re z_l^{n} \\
\Im z_l^{k+2l} & \Im z_l^{k+2l+1} & \dots & \Im z_l^{n} \\
\end{pmatrix}
\begin{pmatrix}
\xi_{k+2l}\\
\vdots\\
\xi_n
\end{pmatrix}
,
\]
we obtain
\begin{equation}\label{2224}
  V(\bx,\bz)J_{\be}(\bx,\bz)+
  \begin{pmatrix}
    A_1 & \mathbf{0} \\
     \mathbf{0} &A_2
  \end{pmatrix}
 =- \begin{pmatrix}
    A_3 & \mathbf{0} \\
     \mathbf{0} & A_4
  \end{pmatrix},
\end{equation}
where
$$
A_1:=
\begin{pmatrix}
\sum_{j=0}^{k+2l-1}j\eta_jx_1^{j-1} &0& \dots &0 \\
0&\sum_{j=0}^{k+2l-1}j\eta_jx_2^{j-1} & \dots &0\\
\vdots &\vdots & \ddots & \vdots \\
0&0& \dots & \sum_{j=0}^{k+2l-1}j\eta_jx_k^{j-1} \\
\end{pmatrix},
$$
$$
A_2:=
\begin{pmatrix}
\sum_{j=0}^{k+2l-1}\eta_j\frac{\partial\Re z_1^j}{\partial\Re z_1}  & \sum_{j=0}^{k+2l-1}\eta_j\frac{\partial\Re z_1^j}{\partial\Im z_1}  &\dots & 0&0\\
\sum_{j=0}^{k+2l-1}\eta_j\frac{\partial\Im z_1^j}{\partial\Re z_1}  & \sum_{j=0}^{k+2l-1}\eta_j\frac{\partial\Im z_1^j}{\partial\Im z_1}  &\dots & 0&0\\
 \vdots &  \vdots&\ddots & \vdots&\vdots\\
 0&0 &\dots&\sum_{j=0}^{k+2l-1}\eta_j\frac{\partial\Re z_l^j}{\partial\Re z_l}  & \sum_{j=0}^{k+2l-1}\eta_j\frac{\partial\Re z_l^j}{\partial\Im z_l} \\
  0&0 &\dots&\sum_{j=0}^{k+2l-1}\eta_j\frac{\partial\Im z_l^j}{\partial\Re z_l}  & \sum_{j=0}^{k+2l-1}\eta_j\frac{\partial\Im z_l^j}{\partial\Im z_l} \\
\end{pmatrix},
$$
$$
A_3:=
\begin{pmatrix}
\sum_{j=k+2l}^{n}j\xi_jx_1^{j-1} &0& \dots &0 \\
0&\sum_{j=k+2l}^{n}j\xi_jx_2^{j-1} & \dots &0\\
\vdots &\vdots & \ddots & \vdots \\
0&0& \dots & \sum_{j=k+2l}^{n}j\xi_jx_k^{j-1} \\
\end{pmatrix},
$$
$$
A_4:=
\begin{pmatrix}
\sum_{j=k+2l}^{n}\xi_j\frac{\partial\Re z_1^j}{\partial\Re z_1}  & \sum_{j=k+2l}^{n}\xi_j\frac{\partial\Re z_1^j}{\partial\Im z_1}  &\dots & 0&0\\
\sum_{j=k+2l}^{n}\xi_j\frac{\partial\Im z_1^j}{\partial\Re z_1}  & \sum_{j=k+2l}^{n}\xi_j\frac{\partial\Im z_1^j}{\partial\Im z_1}  &\dots & 0&0\\
 \vdots &  \vdots&\ddots & \vdots&\vdots\\
 0&0 &\dots&\sum_{j=k+2l}^{n}\xi_j\frac{\partial\Re z_l^j}{\partial\Re z_l}  & \sum_{j=k+2l}^{n}\xi_j\frac{\partial\Re z_l^j}{\partial\Im z_l} \\
  0&0 &\dots&\sum_{j=k+2l}^{n}\xi_j\frac{\partial\Im z_l^j}{\partial\Re z_l}  & \sum_{j=k+2l}^{n}\xi_j\frac{\partial\Im z_l^j}{\partial\Im z_l} \\
\end{pmatrix}.
$$

We finish the proof  by moving in \eqref{2224} the second term from the left-hand side to the right-hand side, using~\eqref{1931}, and noting that for any analytic function $f(z)$
$$
 \det\begin{pmatrix}
   \frac{\partial\Re f}{\partial\Re z} &  \frac{\partial\Re f}{\partial\Im z} \\
      \frac{\partial\Im f}{\partial\Re z} &  \frac{\partial\Im f}{\partial\Im z}
  \end{pmatrix}
  =|f'(z)|^2.
$$
%$$
%\det V(\bx)=\prod_{1\leq i<j\leq k}(x_j-x_i).
%$$
\end{proof}

\begin{lemma}[Coarea formula]\label{lm-coarea}
Let $B\subset \R^m$ be a region. Let $\bu:B\to\R^m$ be a Lipschitz function and $h:\R^m\to\R$ be an $L^1$-function. Then
\begin{equation}\label{1225}
\int\limits_{\R^m}\#\{\bx\in B:\bu(\bx)=\by\}\,h(\by)\dd\by=\int\limits_{B}|\det J_{\bu}(\bx)|\, h(\bu(\bx))\dd\bx,
\end{equation}
where $J_{\bu}(\bx)$ is the Jacobian matrix of $\bu(\bx)$, and $\#S$ denotes the cardinality of a set $S$.
\end{lemma}
\begin{proof}
See~\cite[pp. 243--244]{hF1969}.
\end{proof}

\subsection{Proof \tc{of Theorem~\ref{thm-corr-rel}}}

Now we are ready to finish the proof of Theorem~\ref{thm-corr-rel}. To this end, we  show that  for any family of mutually disjoint Borel subsets $B_1,\dots,B_k\subset\R$ and $B_{k+1},\dots,B_{k+l}\subset\C_+$,
\begin{equation}\label{1647}
\E\left[\prod_{i=1}^{k+l}\mu(B_i)\right]=2^l\int_{B_1}\dots\int_{B_{k+l}}\,\rho_{k+2l}(\bx,\bz)\dd x_1\dots \dd x_k \dd z_1\dots \dd z_l.
\end{equation}
{From~\eqref{1404} we get
\begin{equation}\label{eq-anoth-way}
\E\,\left[\prod_{i=1}^{k+l}\mu(B_i)\right]
=\E\left[\#\{(\bx,\bz)\in B_1\times\dots\times B_{k+l}:\eta(\bx,\bz)=(\xi_0,\dots,\xi_{k+2l-1})\}\right].
\end{equation}}
For clarity, denote the sets $B_{k+j}$ by $\tilde B_{k+j}$ when we consider them as subsets of $\R^2$:
$$
\tilde B_{k+j}:=\{(x,y)\in\R^2:x+\tc{\mathbf{i}}y\in B_{k+j}\}.
$$
Let us apply Lemma~\ref{lm-coarea} {to \eqref{eq-anoth-way}} with
$$
m=k+2l,\quad B=B_1\times\dots\times B_k\times \tilde B_{k+1}\times\dots\times \tilde B_{k+l},
$$
$$
\bu(x_1,\dots,x_{k+2l})=\be(x_1,\dots,x_k,x_{k+1}\pm ix_{k+2},\dots,x_{k+2l-1}\pm ix_{k+2l}),
$$
$$
{h(y_0,\dots,y_{k+2l-1})=f_0(y_0)\dots f_{k+2l-1}(y_{k+2l-1}).}
$$
Note that the indices of $y_i$'s are shifted by 1 according to the {enumeration of
 polynomial coefficients.
Hence due to Lemma~\ref{lm-coarea}, the right-hand side of \eqref{eq-anoth-way} is equal to}
\begin{multline*}
\int\limits_{\R^{k+2l}}\E\left[\#\{(x_1,\dots,x_{k+2l})\in B\,:\, \bu(x_1,\dots,x_{k+2l})=\by\}\right]f_0(y_0)\dots f_{k+2l-1}(y_{k+2l-1})\,\dd\by\\
=\E\int_B|\det J_{\bu}(\bx)|\,\prod_{i=0}^{k+2l-1}f_i(u_i(x_1,\dots,x_{k+2l}))\,\dd x_1\dots \dd x_{k+2l},
\end{multline*}
where  we used {first Fubini's theorem and} then~\eqref{1225}. The Jacobian matrix of $\bu$ coincides with the real Jacobian matrix of $\be$, the determinant of which is given by Lemma~\ref{1200}. Thus,
switching from $\bu$ to $\be$ and again using Fubini's theorem we obtain
\begin{align*}
\E\left[\prod_{i=1}^{k+l}\mu(B_i)\right]&=2^l\int\limits_{B}\E\,\left[\varphi(\bx,\bz)\prod_{i=0}^{k+2l-1}f_i(\eta_i(\bx,\bz))\right]\dd\bx\,\dd\bz,
\end{align*}
{where $\varphi(\bx,\bz)$ is defined in \eqref{eq-phi-def}.}

Combining this with~\eqref{1835} implies~\eqref{1647}, and {due to \eqref{eq-rhokl}} the theorem follows.

\section{Proof of \tc{Theorem}~\ref{1600}}\label{sc-proof-add-ident}

We first consider the case $p<\infty$. Applying~\eqref{1638}    to $G_{p,\bw}$ defined in~\eqref{1956} gives
\begin{align*}
\rho_{p,\bw, n-2l,l}(\bx,\bz)&=\frac{2^{l-n-1} w_0\dots w_n }{\left(\Gamma\left(1+\frac1p\right)\right)^{n+1}}\rv_{n}(\bx,\bz,\bar\bz)\\
&\times\int\limits_{\R }|t|^n \exp\left(-|t|^p\sum_{i=0}^{n}|w_i \sigma_{n-i}(\bx,\bz,\bar\bz)|^p\right)\dd t.
\end{align*}
Using the substitution
$$
s=|t|^p\sum_{i=0}^{n}|w_i \sigma_{n-i}(\bx,\bz,\bar\bz)|^p
$$
and a representation  of the gamma function, we obtain
$$
\rho_{p,\bw, n-2l,l}(\bx,\bz)=\frac{2^{l-n} w_0\dots w_n \Gamma\left(\frac{n+1}{p}\right)}{p\left(\Gamma\left(1+\frac1p\right)\right)^{n+1}}\rv_{n}(\bx,\bz,\bar\bz)
\left(\sum_{i=0}^{n}|w_i \sigma_{n-i}(\bx,\bz,\bar\bz)|^p\right)^{-\frac{n+1}{p}}.
$$
Using the identity
$\frac{1}{p}\, \Gamma\!\left(\frac{n+1}{p}\right) = \frac{1}{n+1}\, \Gamma\!\left(\frac{n+1}{p}+1\right)$
concludes the proof of  Theorem~\ref{1600} for $p<\infty$.

\tc{The case $p=\infty$ follows from the case $p<\infty$ by letting $p\to\infty$:
\begin{align*}
\rho_{\infty,\bw, n-2l,l}(\bx,\bz)&=\lim_{p\to\infty}\rho_{p,\bw, n-2l,l}(\bx,\bz)\\
&=\frac{2^{l-n} w_0\dots w_n \rv_{n}(\bx,\bz,\bar\bz) }{(n+1)(\max_{0\le i\le n}|w_i\sigma_{n-i}(\bx,\bz,\bar\bz)|)^{n+1}},
\end{align*}
where in the first equality we used the continuity of the $\Gamma$-function \tb{at  1}  and in the second -- the limit equality
\[
\lim_{p\to\infty} \left(\sum_{j=1}^m |b_j|^p\right)^{1/p}=\max_{1\le j\le m} |b_j|.
\]
The theorem follows.}

\section*{Acknowledgments}

The authors are grateful to Zakhar Kabluchko and Manjunath Krishnapur for informing the authors  about a number of relevant papers in this research area. \tc{The authors also wish to express their thanks to the unknown referee for many suggestions which improved the paper and, in particular,  simplified the original proof of Theorem~\ref{1600}.}
%The authors would like to thank the referees from IMRN for their careful reading the paper and valuable remarks.

\bibliographystyle{abbrv}
\bibliography{corrf5}

\end{document}